\documentclass[11pt, reqno]{amsart}    

\usepackage{amsmath,amssymb,latexsym}
\usepackage{graphicx,xcolor,mathrsfs}
\usepackage{cite}

\usepackage[font={scriptsize}]{caption}

\newcommand{\F}{\mathcal{F}}

\newcommand{\R}{\mathbb{R}}

\newcommand{\I}{\mathrm{i}}
\newcommand{\bigO}{\mathcal{O}}
\newcommand{\diff}{\mathrm{d}}

\DeclareMathOperator{\sign}{sgn}

\newtheorem{theorem}{Theorem}[section]
\newtheorem{lemma}[theorem]{Lemma}
\newtheorem{proposition}[theorem]{Proposition}
\newtheorem{remark}[theorem]{Remark}

\newtheorem{corollary}[theorem]{Corollary}
\newtheorem*{main-theorem}{Main Theorem}
\newtheorem*{remark*}{Remark}

\numberwithin{equation}{section}

\usepackage[colorlinks=true]{hyperref}
\hypersetup{urlcolor=blue, citecolor=red, linkcolor=blue}

\begin{document}

\allowdisplaybreaks

\title[Enhanced existence time in fKdV]{Enhanced existence time of solutions to the fractional Korteweg--de~Vries equation}

\author{Mats Ehrnstr\"om}
\author{Yuexun Wang}

\address{Department of Mathematical Sciences, Norwegian University of Science and Technology, 7491 Trondheim, Norway.}
\email{mats.ehrnstrom@ntnu.no}
\email{yuexun.wang@ntnu.no}

\thanks{The authors acknowledge the support by grant nos. 231668 and 250070 from the Research Council of Norway.}

\subjclass[2010]{76B03, 35S30, 76B15}
\keywords{Enhanced life span, fKdV, dispersive equations}

\begin{abstract}
We consider the fractional Korteweg--de~Vries equation \(u_t + u u_x - |D|^\alpha u_x = 0\) in the range of \(-1<\alpha<1\), \(\alpha\neq0\).   
Using basic Fourier techniques in combination with the modified energy method we extend the existence time of classical solutions with initial data of size $\varepsilon$ from \(\frac{1}{\varepsilon}\) to a time scale of \(\frac{1}{\varepsilon^2}\). This analysis, which is carried out in Sobolev space \(H^N(\mathbb{R})\), \(N \geq 3\), answers positively a question posed by Linares, Pilod and Saut in \cite{MR3188389}.
\end{abstract}
\maketitle


\section{Introduction}
We consider the fractional Korteweg--de~Vries (fKdV) equation
\begin{align}\label{eq:fKdV}
\partial_tu+u\partial_xu-|D|^\alpha\partial_xu=0,
\end{align}
where the parameter \(\alpha\) may in general take any real value. Here,  \(u\colon [0,T]\times \R \mapsto \R\) and
\begin{align*}\label{eq:|D|}
\widehat{|D|^\alpha f}(\xi)=|\xi|^\alpha\widehat{f}(\xi).
\end{align*}
More generally, under the Fourier transform 
\[
\mathcal{F}(f)(\xi)=   \int_{\R} f(x) \exp(-\mathrm{i}\xi x) \, \diff x,
\]
we let \( D = -\I\partial_x\) and denote by  $\sigma(D)$ the Fourier multiplier operator defined from its symbol \(\sigma(\xi)\) via the relation
\begin{eqnarray*}
	\mathcal{F}(\sigma(D)f)(\xi)=\sigma(\xi)\widehat{f}(\xi).
\end{eqnarray*}

\subsection{The fKdV family}
When \(\alpha=2\) and \(\alpha=1\), the fKdV equation~\eqref{eq:fKdV} reduces to the classical Korteweg--de~Vries (KdV) and Benjamin--Ono (BO) equations, respectively. For \(\alpha = 0\) one obtains the inviscid Burgers equation, being the only non-dispersive member in the fKdV family. When \(\alpha = -1\) one instead has the Burgers--Hilbert (BH) equation, and when \(\alpha = -2\) the reduced Ostrovsky (RO) equation. All in all, the fKdV family has been suggested as scale for investigating the balance of nonlinear and dispersive effects \cite{MR3188389}, especially when the dispersion is very weak, meaning \(\alpha\) takes small or negative values.\footnote{Note that the fKdV equation \eqref{eq:fKdV}, which is dispersive, is inherently different from the dissipative equation \(u_t + u u_x + |D|^{\beta} u = 0\), investigated for example in \cite{MR2455893}.}

To quantify this, note that both the KdV (\(\alpha = 2\)) and BO (\(\alpha = 1\)) equations are globally well-posed in Sobolev space \(H^s(\R)\). KdV is globally well-posed in \(H^{-1}(\R)\) \cite{MR2267286,2018arXiv180204851K}, and BO in \(L^2(\R)\) \cite{MR2291918, MR2970711}. For the fractional values \(\alpha \in (1,2)\) one has also global well-posedness in \(L^2(\R)\) \cite{MR2754070} (but see also \cite{MR3103170} for a result in weighted Sobolev spaces). For \(\alpha\) below unit order, Molinet, Pilod and Vento \cite{MR3906854} very recently established global well-posedness  in \(H^{\frac{\alpha}{2}}(\R)\) with \(\alpha>\frac{6}{7}\).  Numerical simulations from \cite{MR3317254} however suggest global well-posedness for all \(\alpha > \frac{1}{2}\). The value \(\alpha = \frac{1}{2}\) is scaling critical and believed to be critical also for the global well-posedness theory \cite{MR3188389}. For values of \(\alpha\) less than \(\frac{1}{2}\) there are only partial results or results under constraints. This is due to the presence of smooth solutions that blow-up in finite time (\(C^{1+\varepsilon}\)-blowup for \(\alpha \in [-1,0)\) \cite{MR2727172} and wave-breaking for \(\alpha \in (-1,-\frac{1}{3})\) \cite{MR3291137}) and, connected thereto, non-uniqueness issues that appear over longer times \cite{MR3248030}. To remedy this, one may either as in \cite{MR3248030} turn to more feasible (weak) solution concepts, or restrict attention to a subclass of initial data for which blow-up is excluded \cite{MR3094592}. Note that, still, both these works are for integer values of \(\alpha\), in which case the equation has a straightforward interpretation on the physical side (local in the case of \(\alpha = -2\), and involving the Hilbert transform in the case \(\alpha = -1\)).

\subsection{Long-time existence}
Because of the above difficulties for low sub-unit values of \(\alpha\), a question of interest is that of long-time existence. This question was raised specifically in \cite[Remark 4.5]{MR3188389}, and pointed out to one of the authors by the authors of that paper\footnote{The main question raised in \cite{MR3188389} concerns existence of solutions to \(u_t + \varepsilon u u_x -\varepsilon |D|^{\alpha} u_x = 0\) with initial data of unit size, which for long-time existence is rather different than \(u_t + \varepsilon u u_x - |D|^{\alpha} u_x = 0\). The latter is the equation considered in Section 4 of \cite{MR3188389} as well as in our paper.}. At that time, there was already a proof for the integer case \(\alpha = -1\) of the Burgers--Hilbert equation \cite{MR2982741}, and in fact there are two \cite{MR3348783}, but no results for the fractional cases \(\alpha \in (-1,1)\), \(\alpha \neq 0\). Although both the KdV and the Benjamin--Ono equations model water waves in specific regimes, the range \(\alpha \in (-1,1)\) in \eqref{eq:fKdV} is extra interesting as it corresponds to full-dispersion models for capillary (\(\alpha = \frac{1}{2})\) and gravity \((\alpha = -\frac{1}{2})\) waves on deep water \cite{MR3188389}. In particular, the case \(\alpha = -\frac{1}{2}\) is the homogeneous equivalent of the inhomogeneous Whitham equation (see \cite{EW16} for a fairly complete list of references for that equation), that has received quite a bit of attention lately, and wherefrom our interest comes. In fact, the method here developed is amenable to a generalisation covering a class of nonlinear dispersive equations with inhomogeneous symbols which allows for more singular interactions in low frequencies than in the homogeneous case, say, the Whitham equation, and we hope to pursue that in a forthcoming investigation.

Coming back to \cite{MR3188389} and \cite{MR2982741}, the question raised is whether classical solutions in \(H^s(\R)\) of initial size \(\varepsilon\) may be extended, in \(H^s(\R)\), beyond the hyperbolic existence time \(\bigO(1/\varepsilon)\). The standard energy estimates \cite{MR533234} on the equation \eqref{eq:fKdV} namely give  
\begin{equation}\label{1.2}
\frac{\diff}{\diff t} \| \partial_x^ku \|_{L^2}^2\lesssim \|u_x \|_{L^{\infty}} \|u \|_{H^k}^2, 
\end{equation}
which  yields only the existence of classical solution on a time scale \(\bigO(\frac{1}{\varepsilon})\).  As developed in \cite{MR3348783} for the case \(\alpha = -1\), an improvement of this may be achieved using a modified energy based on the normal form, but without switching to normal-form variables. In the context of the Burgers--Hilbert equation, this leads to a short proof of the desired existence on a time scale \(\bigO(1/\varepsilon^2)\) by working directly on the physical side in space and using estimates for the Hilbert transform. We, too, use a normal-form transform to construct a modified energy as in \cite{MR3348783}, but the fractional values of \(\alpha\) makes for substantial differences in the remaining part of the proof. In particular, the normal transform we use involves a pseudo-product related directly to the water-wave problem, and we work solely on the Fourier side to obtain the desired energy estimates by using basic \(L^2\)- and \(L^\infty\)-estimates in cubic and quartic expressions. The details of this will be explained below. We shall comment that our work uses only the normal form and the modified energy, but not the decay in time of solutions, and the fKdV equation might be globally well-posed for localised small data when \(0 < \alpha < 1/2\). Such a proof, however, probably would require a better understanding of the cancellations of the nonlinearity in low frequencies, and call for some new techniques.

\subsection{Related works} Normal forms have a long history in mathematics, often given the names of Poincar\'e or Birkhoff for their works \cite{Poincare} and \cite{MR1555257} on the topic. In modern PDEs normal forms are most often connected with Shatah (see \cite{MR803256}), and appear naturally when working with water waves and Hamiltonian formalism for such \cite{MR3502161,MR3445499}. There are clever ways to modify these transforms to deal with the loss of derivatives that may accompany them \cite{MR3460636}, and a different way is to modify the energy as done in \cite{MR3348783}. We follow the latter idea, adopting the normal transform simply to our fractional case by introducing a pseudo-product as ansatz. This pseudo-product then influences the methods used in the estimates that form the bulk of the paper. The method from \cite{MR3348783}  has been further developed in collaboration with several different authors in a series of papers, where we mention \cite{MR3535894}, \cite{MR3625189} and \cite{MR3667289} as they are most closely related to our results. It is difficult to compare them directly: the water-wave problem is clearly more involved in its original formulation and requires a lot of work just to deal with the normal transform; on the other hand, the exact relation between the cases \(\alpha = \pm \frac{1}{2}\) for the fKdV equation \eqref{eq:fKdV} and the water-wave problem is, as far as we know, not formalised on these time scales. In working with the water-wave problem the authors of \cite{MR3535894,MR3625189,MR3667289} use holomorphic coordinates to reduce the initial equation to a simpler form. It is not unthinkable that in the deep-water case one could instead relate it to the fKdV, but this remains open.

Concerning our method and the bulk of the paper it is, probably more closely related to the work \cite{MR3348783} on the Burgers--Hilbert equation, as well as the paradifferential and  \(L^2/L^\infty\)-estimates appearing in long-term Sobolev analyses of the water-wave problem such as \cite{MR3460636}, and dispersive problems such as \cite{MR3552008} (there the authors treat a cubic and complex-valued scalar equation qualitatively similar to the capillary case). We prove equivalence of the modified and classical energy, just as in \cite{MR3348783}, but a commutator that completely vanishes in the case \(\alpha = -1\) leaves high-order terms in the fractional case. In contrast to the Burgers--Hilbert case, we also work completely on the Fourier side, encountering symbols with singularities of the same form as when using time-space frequency analysis for the water-wave problem \cite{MR2993751} (see also the introduction of \cite{wang2016global}, that shows the connection to \eqref{6} quite well). Just as in these other works, Coifman--Meyer estimates cannot be used because of the frequency interactions/singularities appearing in the symbols. It is exactly these frequency interactions that arise from the dispersive nature of the problem. Instead, one relies on \(L^2 \times L^2 \times L^\infty\) and \(L^2 \times L^2 \times L^\infty \times L^\infty\)-estimates which regain the equivalence of Coifman-Meyer theory for the problem at hand. In contrast to most other works, which make use of paradifferential calculus in attacking frequency space, we will simply divide frequency space in a rudimentary way. This is not entirely enough, and finally we rely on a global transformation in frequency space to obtain an order-reducing commutator. The main steps in our analysis and the division of frequency space are outlined in Figure~\ref{fig:outline} and will be described more exactly below. 

To state our result, let \(H^s(\R) = (1-\partial_x^2)^{-s/2} L^2(\R)\) be the standard Bessel-potential (Sobolev) spaces and, for any Banach space \(\mathbb Y\),  let \(C^k([0,T];\mathbb{Y})\) be the Banach space of all bounded continuous functions \(u\colon [0,T]\rightarrow \mathbb{Y}\) with bounded and continuous derivatives up to \(k\)th order. We write \(f\lesssim g\) when \(f/g\)  is uniformly bounded from above, and \(f \eqsim g\) when \(f \lesssim g \lesssim f\).

\begin{theorem}\label{thm:main}   
	Let \(-1<\alpha<1\)\footnote{In the subcritical case \(\alpha > \frac{1}{2}\) a scaling argument is enough to conclude that solutions exist on the time scale \(\bigO(1/\varepsilon^2)\) \cite{MR1976047}. We state our results in terms of  \(\alpha \in (-1,1)\), \(\alpha \neq 0\), as this is natural and simple.}, \(\alpha\neq 0\) and \(N\geq 3\). There exists \(\varepsilon_* > 0\), such that
	for any initial data satisfying
	\begin{align*}
		\|u_0\|_{H^N(\R)}\leq\varepsilon,
	\end{align*}
	with \(\varepsilon \leq\varepsilon_*\), there exist a positive number \(T\gtrsim \frac{1}{\varepsilon^2}\) and a unique solution \(u\) in \(C([0,T];H^N(\R))\cap C^1([0,T];H^{N-2}(\R))\) of  \eqref{eq:fKdV} with \(u(0,x)=u_0(x)\) such that 	
	\begin{align*}
		\|u\|_{C([0,T];H^N(\R))}\lesssim \varepsilon.
	\end{align*}
\end{theorem}

\subsection{Outline}
As described above, we start in {\bf Section~\ref{sec:normal form}} by finding a normal-form transformation to remove the quadratic term \(u u_x\) from the equation \eqref{eq:fKdV}, which leads to a resulting  equation with  a cubic nonlinearity. As this transformation involves a not-so-straightforward and singular Fourier symbol \(m \colon \R^2 \to \R\), in Proposition~\ref{proposition:multiplier 1} we give a global growth characterisation of it for \(\alpha \in (0,1)\) and \(\alpha \in (-1,0)\), respectively. This growth characterisation is not completely sufficient, as we need to use symmetries and the exact form of \(m\) later, but it simplifies a lot of estimates. It is worth noting that the behaviour of \(m\) is qualitatively different for positive and negative values of \(\alpha\).

  In {\bf Section~\ref{sec:modified energy}} we follow \cite{MR3348783} to define a modified energy based on the normal-form variables (here, one can neglect quartic terms as these are negligible in relation to the cubic for small data), and then show that the modified energy is equivalent to the standard  \(H^s\)-energy if the latter  is small. To prove this we use the skew-symmetry property \eqref{6.1} of the symbol \(m\), corresponding to a generalised integration by parts in two variables, combined with \(L^2/L^\infty\)-estimates using Proposition~\ref{proposition:multiplier 1}. These estimates are cubic, and not very difficult.

  In {\bf Section~\ref{sec:energy estimates}} we perform the quartic energy estimates on the modified energy from Section~\ref{sec:modified energy}. It is here that the main differences to \cite{MR3348783} and the other work cited above appear, and an attempt to illustrate our approach has been made in Figure~\ref{fig:outline}. \emph{The general strategy is to (i) estimate lower-order terms using pointwise Fourier-estimates based on the growth of the symbol \(m\), (ii) try to eliminate the highest-order terms by (a) the use of generalised integration by parts (skew-symmetries), (b) dividing up frequency space, and (c) by global transformations (changes of variables on the Fourier-side).}  The precise steps are as follows.
 \begin{figure}
\includegraphics[width=0.8\textwidth]{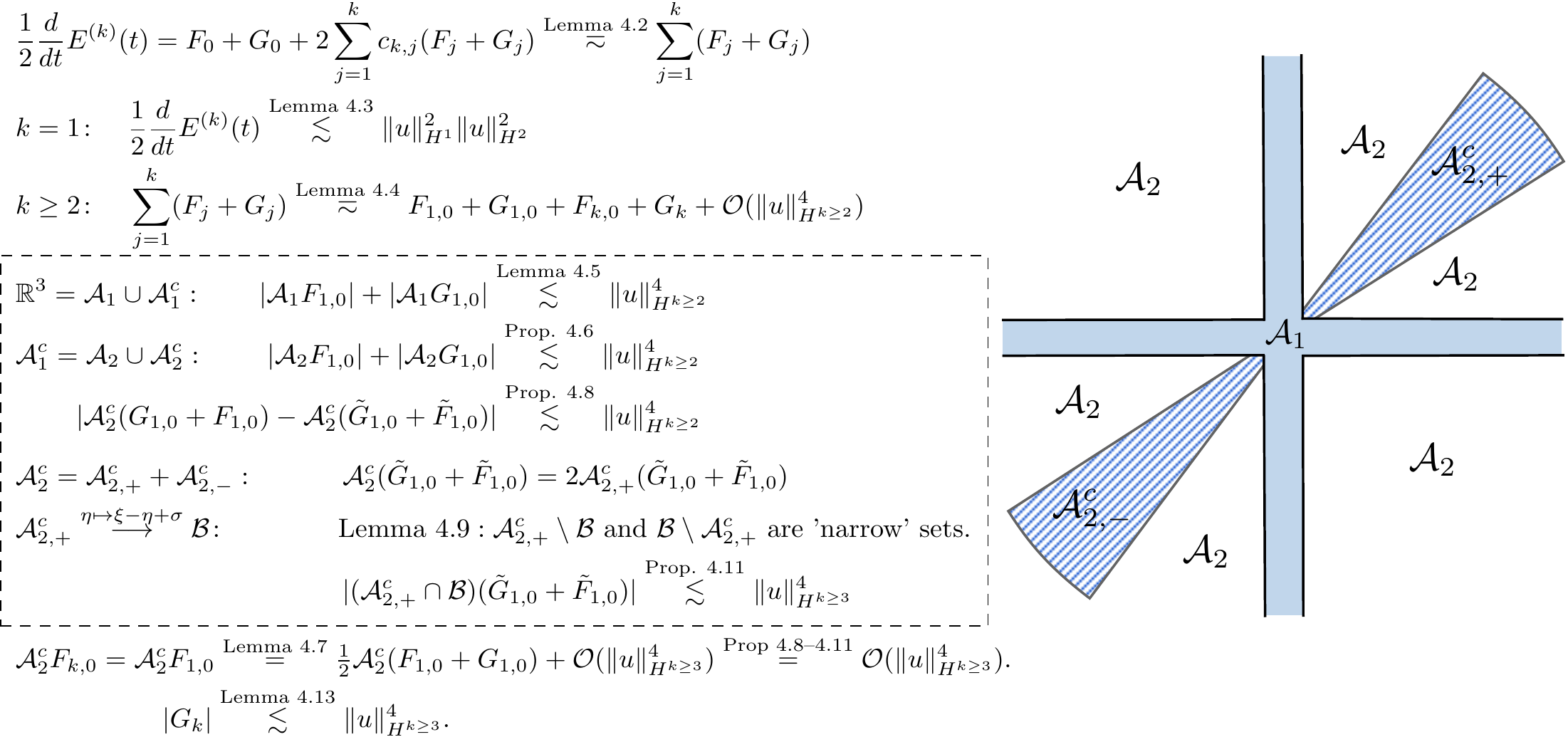}
\caption{An schematic outline of the main steps in the energy estimates carried out in Section~\ref{sec:energy estimates}. The frequency sets \(\mathcal A\) have been qualitatively sketched in the plane as to enhance their visibility (in reality they are subsets of three-dimensional space). The most subtle part of the energy estimates is highlighted: the main commutator is handled in Proposition~\ref{prop:G1 + F1 full}, based on the change of variables \(\eta - \sigma \to \xi - \eta\), which maps the set \({\mathcal A}_{2,+}^c\) to the set \(\mathcal B\) and leaves the measure \(\diff Q(\hat u)\) from \eqref{eq:dQ} invariant.}
\label{fig:outline}
\end{figure}
 We first make away with the very highest-order terms by the use of cancellation in a commutator in Lemma~\ref{lemma:F_0 G_0}. Note that this cancellation appears also for the the case \(\alpha = -1\), but in our case high-order terms still remain. We then treat the first partial energy \(E^{k}|_{k=1}\) separately in Lemma~\ref{lemma:k=1}, and estimate the lower-order terms in Lemma~\ref{lemma:ddt sim I} using the global bounds on the symbol \(m\); this leaves us with four higher-order terms, two of which can later be controlled by the others. At this point we split frequency space: in a low-frequency part \(\mathcal{A}_1\) concentrated around the axes of \((\xi, \eta, \sigma)\)-space (Lemma~\ref{lemma:A_1 estimates}), and in a large part \(\mathcal{A}_2\) where we are free to move derivatives from the high-order to the low-order terms (Lemma~\ref{prop:A_2}). After some minor adjustment to the terms in Proposition~\ref{prop:tilde FG}, and noticing a symmetry, we finally need to deal with the terms in a narrow and positive cone \(\mathcal{A}_{2,+}^c\). In this set \(\xi \eqsim \eta \eqsim \sigma \geq 1\) and pointwise Fourier estimates cannot be used to estimate the terms directly (negative values of \(\alpha\) near zero appear hardest). We therefore make a change of variables in Fourier space to one of the two terms, which \emph{does not leave the set \(\mathcal{A}_{2,+}^c\) invariant}. The set difference, however, is shown to be negligible in the estimates in Lemma~\ref{lemma:sets}, and the resulting commutator is a good one: by Taylor expanding the symbol in the two small variables \(\frac{\xi - \eta}{\eta}\) and \(\frac{\sigma - \eta}{\eta}\) in a subset of the positive cone we get two orders of cancellation and can close our estimates in Proposition~\ref{prop:G1 + F1 full}. The remaining two terms (of the total four from the beginning) can then be controlled using Lemma~\ref{lemma:F10 sim G10}.

The short proof of the main result is given in {\bf Section~\ref{sec:final}}.

\section{The normal-form transformation}\label{sec:normal form}
In the spirit of \cite{MR803256}, we introduce the normal-form transformation \(u\mapsto w\) as follows 
\begin{align}\label{3.5}
	w=u+P(u,u),
\end{align}
but where we now seek a bilinear form \(P\)  defined as a pseudo-product 
\begin{align}\label{4}
	\mathcal{F}(P(f_1,f_2))(\xi)=\int_{\R}m(\xi-\eta,\eta)\hat{f}_1(\xi-\eta)\hat{f}_2(\eta)\,\diff \eta.
\end{align}
One may of course similarly write \(m(\xi,\eta)\), but we prefer to use the variables \(\xi-\eta\) and \(\eta\) for reasons that shall hopefully soon be clear. The normal-form transformation \eqref{3.5} will be uniquely determined by the multiplier \(m(\xi-\eta,\eta)\). Note that \(P(f_1,f_2)\) is symmetric in \(f_1\) and \(f_2\) if and only if the multiplier \(m(\xi-\eta,\eta)\) is symmetric in \(\xi-\eta\) and \(\eta\).
It will be convenient for later use to write \(P(f_1,f_2)\) as a  symmetric form. For this, we write \(\mathcal{F}(u\partial_xu)(\xi)\) as 
\begin{align}\label{4.5}
	\mathcal{F}(u\partial_xu)(\xi)=\frac{1}{2}\int_{\mathbb{R}}(\mathrm{i}\xi)\hat{u}(\xi-\eta)\hat{u}(\eta)\, \diff \eta.
\end{align} 
Direct calculations show
\begin{equation}\label{5}
	\begin{aligned}
		\partial_tw-|D|^\alpha\partial_xw &= \partial_tu-|D|^\alpha\partial_xu + P(\partial_tu,u) +P(u,\partial_tu)\\
		&\quad - |D|^\alpha\partial_xP(u,u)\\
		&= -u\partial_xu +P(|D|^\alpha\partial_x u-u\partial_x u,u) + P(u,|D|^\alpha\partial_x u-u\partial_x u)\\
		&\quad - |D|^\alpha\partial_xP(u,u)\\
		&= -u\partial_xu + P(|D|^\alpha\partial_x u,u) + P(u,|D|^\alpha\partial_x u)\\
		&\quad -|D|^\alpha\partial_xP(u,u) + R(u),
	\end{aligned}
\end{equation}
where \(R(u)\) is a cubic nonlinearity. In light of \eqref{4}, one has  
\begin{equation}\label{5.5}
	\begin{aligned}
		&\mathcal{F}\big(P(|D|^\alpha\partial_x u,u)\big)(\xi)= \I \int_{\R} (\xi-\eta) |\xi-\eta|^\alpha  m(\xi-\eta,\eta)\hat{u}(\xi-\eta)\hat{u}(\eta) \, \diff\eta,\\
		&\mathcal{F}\big(P(u,|D|^\alpha\partial_x u)\big)(\xi)= \I \int_{\R} \eta |\eta|^\alpha  m(\xi-\eta,\eta)\hat{u}(\xi-\eta)\hat{u}(\eta)\, \diff \eta,\\
		&\mathcal{F}\big(|D|^\alpha\partial_xP(u,u)\big)(\xi)= \I \int_{\R} \xi |\xi|^\alpha m(\xi-\eta,\eta)\hat{u}(\xi-\eta)\hat{u}(\eta)\, \diff \eta.
	\end{aligned}
\end{equation}
It follows from \eqref{4.5}, \eqref{5} and \eqref{5.5} that
\begin{equation}\label{5.6}
	\begin{aligned}
		&\mathcal{F}(\partial_tw-|D|^\alpha\partial_xw)(\xi)\\
		&=\I \int_{\R}\left( -\frac{\xi}{2}+m(\xi-\eta,\eta)\big[(\xi-\eta) |\xi-\eta|^\alpha+ \eta |\eta|^\alpha
		- \xi |\xi|^\alpha\big]\right)\\
		&\quad\times\hat{u}(\xi-\eta)\hat{u}(\eta)\, \diff \eta+\F(R(u))(\xi).
	\end{aligned}
\end{equation}
In order to remove the quadratic nonlinearity in the  equation of  \eqref{5.6}, we should set
\begin{align}\label{5.8}
	-\frac{\xi}{2}+m(\xi-\eta,\eta)\big[(\xi - \eta) |\xi-\eta|^\alpha + \eta |\eta|^\alpha - \xi|\xi|^\alpha\big]=0,
\end{align}
meaning
\begin{align}\label{6}
	m(\xi-\eta,\eta)=\frac{\xi}{2\big[|\xi-\eta|^\alpha(\xi-\eta)+|\eta|^\alpha\eta-|\xi|^\alpha\xi\big]}.
\end{align}
Thus \(m\) is symmetric in \(\xi - \eta\) and \(\eta\), and
\begin{align}\label{6.1}
	m(\xi-\eta,\eta)\eta+m(\eta-\xi,\xi)\xi=0.
\end{align} 
Note that for general variables \(m(a,b)\), 
\begin{align}\label{6.25}
	m(a,b)=\frac{a+b}{2\big[|a|^\alpha a+|b|^\alpha b-|a+b|^\alpha (a+b)\big]}.
\end{align}
With the above choise of \(m\) the equation \eqref{5.6} reduces to 
\begin{align}\label{6.2}
	\partial_tw-|D|^\alpha\partial_xw =R(u).
\end{align}  
We can now use \eqref{6} to recover the formula \eqref{3.5} in the case of \(\alpha=-1\) which was first obtained by Biello and Hunter \cite{MR2599457} by a symplectic near-identity transformation based on a Lie-series expansion of the corresponding Hamiltonian. Indeed,
\eqref{6} yields
\begin{align*}
	m(\xi-\eta,\eta)
	={\textstyle \frac{1}{2}}\xi\sign(\xi)\sign\big((\xi-\eta)\eta\big),
\end{align*}
meaning
\begin{align*}
	P(u,u)={\textstyle \frac{1}{2}}|D|(\mathbb{H}u)^2.
\end{align*}
Thus one has
\begin{align}\label{6.3}
	w=u+{\textstyle \frac{1}{2}}|D|(\mathbb{H}u)^2.
\end{align}
Similarly, one can also deduce  the normal-form transformation  when  \(\alpha=2\) 
(cf. \cite{MR2982741}), as
\begin{align}\label{6.5}
	w=u- {\textstyle \frac{1}{6}} (\partial_x^{-1}u)^2.
\end{align}

\begin{proposition}\label{proposition:multiplier 1} 
	The following relations hold globally in \((\xi,\eta)\)-space. 
	\begin{equation}\label{6.8a}
		|m(\xi-\eta,\eta)| \eqsim \frac{|\xi-\eta|^{1-\alpha}}{|\eta|}+\frac{|\eta|^{1-\alpha}}{|\xi-\eta|},
	\end{equation}
	for \(0<\alpha<1\). When \(0<\beta=-\alpha<1\), one instead has
	\begin{equation}\label{6.8b}
		|m(\xi-\eta,\eta)| \eqsim  |\xi|^\beta \left(\frac{|\xi-\eta|^{1-\beta}}{|\eta|^{1-\beta}}
		+\frac{|\eta|^{1-\beta}}{|\xi-\eta|^{1-\beta}}\right).
	\end{equation}
\end{proposition}

\begin{remark} 
	The formulas \eqref{6.8a} and \eqref{6.8b} are valid also outside of our range of interest \(\alpha \in (-1,1)\), \(\alpha \neq 0\). Note that they imply that the operator \(P\) is not invertible at low frequencies, and possesses extra derivatives at high frequencies. Sections~\ref{sec:modified energy} and~\ref{sec:energy estimates} will mainly be devoted to treating the difficulties this brings in the energy estimates.      
	
\end{remark}

\begin{proof}  For \(r\geq0\), \(0\leq\theta<2\pi\), let	
	\begin{align*}
		\xi-\eta=r\cos\theta,\quad \eta=r\sin\theta,
	\end{align*}
	\begin{align*}
		\mu(\theta)=|\cos\theta|^\alpha\cos\theta+|\sin\theta|^\alpha\sin\theta
		-|\cos\theta+\sin\theta|^\alpha(\cos\theta+\sin\theta)
	\end{align*}
	and
	\begin{align*}
		\lambda(\theta)=\frac{\cos\theta+\sin\theta}{\mu(\theta)}.
	\end{align*}
	Then it follows from \eqref{6} that
	\begin{align*}
		m(\xi-\eta,\eta)=m(r\cos\theta,r\sin\theta)=\frac{\lambda(\theta)}{2r^{\alpha}}.
	\end{align*}
	For all \(\alpha \in (-1,1)\), \(\alpha\neq 0\), we then have  \(\mu(\theta)=0\) if and only if either of the three terms \(\cos\theta\), \(\sin\theta\) or \(\sin\theta + \cos\theta\) vanishes. Thus the six zeros \(0, \frac{1}{2} \pi,\frac{3\pi}{4}, \pi, \frac{3}{2} \pi, \frac{7\pi}{4}\) are the only zeros of \(\mu\),  and by Taylor's formula and L'Hospital's law one easily calculates that they are single (of order exactly 1) when \(\alpha \in (0,1)\). Since \(\mu\) is continuous, one therefore has  
	\begin{align*}
		\mu =(\cos+\sin)(\cos)(\sin) \, h,
	\end{align*}
	where \(h\)  is a bounded  function that is also bounded away from zero (but not necessarily continuous).		
	We conclude that
	\begin{align*}
		m(\xi-\eta,\eta)&=\frac{\cos\theta+\sin\theta}{2r^\alpha(\cos\theta+\sin\theta)\cos\theta\sin\theta \, h(\theta)}\\
		&=\frac{r^{2-\alpha}}{2(r\cos\theta) (r\sin\theta) \, h(\theta)},
	\end{align*}
	which  finishes the proof since \(h\) is bounded away from both zero and infinity.
	
	If \(-1<\alpha<0\) then, by Taylor's formula and L'Hospital law, one instead calculates that the order of the above zeros of \(\mu\) is \(1+ \alpha\), meaning
	\begin{align*} 	
		\lim_{\theta\rightarrow \theta_*}\frac{|\mu(\theta)|}{|\sin\theta|^{1+\alpha}} \in (0,\infty), \qquad \theta_* \in \{0, \pi\},
	\end{align*}	
	and similarly with \(\sin \theta\) replaced by either \(\cos\theta\) or \(\cos\theta + \sin \theta\) for their respective zeros. By the same argument as above, 
	\begin{align*}
		\mu &=|\cos+\sin|^{1+\alpha} |\cos|^{1+\alpha} |\sin|^{1+\alpha} \tilde{h},
	\end{align*}
	where \(\tilde{h}\) is bounded as well as bounded away from zero (but not continuous). Hence,
	\begin{align*}
		|m(\xi-\eta,\eta)| &=\frac{|\cos\theta+\sin\theta|}{2r^\alpha|\cos\theta+\sin\theta|^{1+\alpha} |\cos\theta|^{1+\alpha} |\sin\theta|^{1+\alpha} |\tilde{h}(\theta)|}\\
		&=\frac{r^{2+2\alpha}}{|r(\cos\theta+\sin\theta)|^{\alpha} |r \cos\theta|^{1+\alpha} |r\sin\theta|^{1+\alpha} |\tilde{h}(\theta)|},
	\end{align*}
	which  completes the proof.
\end{proof}

\section{The modified energy}\label{sec:modified energy}
Standard theory \cite{MR533234} can be used to show that there exists a positive number \(T\gtrsim \frac{1}{\varepsilon}\) and a unique  solution \(u\in C([0,T];H^N(\R))\) of \eqref{eq:fKdV}. Therefore, to prove Theorem \ref{thm:main}, we need only to prove an a priori \(H^N(\R)\)-bound for classical solution \(u\in C([0,T];H^N(\R))\). 
It follows from \eqref{6.2} that \(w\) obeys  
\begin{align}\label{eq:modified by norm form}
	\partial_tw-|D|^\alpha\partial_xw
	= R(u),
\end{align}
with \(R(u)\) a cubic nonlinearity. There is a loss of derivatives if the standard energy estimates are applied directly to \eqref{eq:modified by norm form}. 
To get around this difficulty, one possible way is to work on the original equation \eqref{eq:fKdV} with a modified energy that can captures the interplay between dispersive and nonlinear effects. This is the technique from \cite{MR3348783}. In light of \eqref{3.5} one calculates that      
\begin{align}\label{7}
	\|\partial_x^kw\|_{L^2}^2=\|\partial_x^ku\|_{L^2}^2+2\big(\partial_x^ku, \partial_x^kP(u,u)\big)_2+\|\partial_x^kP(u,u)\|_{L^2}^2,
\end{align} 
where we use  \((f,g)_2=\int_{\R}fg \, \diff x\) to denote the inner product in  \(L^2(\R)\). As observed in \cite{MR3348783}, the last term on the right hand side of \eqref{7} is irrelevant to the cubically nonlinear energy estimates since it is quartic. This suggests removing that term from \eqref{7}, while keeping the other two as part of a modified energy
\begin{align}\label{eq:modified k-energy}
	E^{(k)}(t)=\|\partial_x^ku\|_{L^2}^2+2\big(\partial_x^ku, \partial_x^kP(u,u)\big)_2.
\end{align}
For data that is small in \(H^N(\R)\) the modified energy is almost equivalent to the Sobolev energy. To handle the low-frequency singularities in the symbol \(m\), however, we shall use only \(E^{(k)}\) starting from \(k=1\), adding a simple \(L^2\)-term to our total energy to capture the inhomogeneous (zeroth order) part of a solution.

\begin{lemma}\label{lemma:lower bound  estimates}  
Let \(\alpha \in (-1,1)\), \(\alpha \neq 0\). There then exists \(\varepsilon > 0\) such that
\begin{align}\label{8}
  \sum_{k=1}^NE^{(k)}(t)+\|u\|_{L^2}^2 \eqsim \|u\|_{H^N}^2,
\end{align}
uniformly for \(\|u\|_{H^N} < \varepsilon\).
\end{lemma}

\begin{proof} Let \(k\geq1\). We will be done if we can show that 
\begin{equation}\label{eq:cubic estimate}
\big(\partial_x^ku, \partial_x^kP(u,u)\big)_2 \lesssim \varepsilon \|u\|_{H^k}^2,
\end{equation}
whenever \(\|u\|_{H^N} < \varepsilon\). We first use the symmetry of \(P\) to write  
	\begin{align*}
		&\big(\partial_x^ku, \partial_x^kP(u,u)\big)_2\\
		&=2\big(\partial_x^ku, P(u,\partial_x^ku)\big)_2+\sum_{j=1}^{k-1} c_{k,j} \big(\partial_x^ku, P(\partial_x^ju,\partial_x^{k-j}u)\big)_2\\
		&=\colon 2A_0+\sum_{j=1}^{k-1}A_j,
	\end{align*}
where \(A_0\) is the worst term in view of Proposition~\ref{proposition:multiplier 1}. The precise structure of \(m\), however, allows us to treat it using integration by parts as follows. Note that on the Fourier side, integration by parts corresponds to formula \(-\I\xi = -(\I(\xi - \eta) + \I\eta)\), where the minus signs in front of \(\I\xi\) comes from the complex conjugate in the inner product. Hence, 
\begin{equation}\label{10}
\begin{aligned}
A_0 &=\iint_{\mathbb{R}^2}m(\xi-\eta,\eta)\hat{u}(\xi-\eta) (\I \eta)^k \hat{u}(\eta) \overline{(\I \xi)^k \hat{u}(\xi)}\,\diff\eta\, \diff \xi\\
	&=-\iint_{\mathbb{R}^2}m(\xi-\eta,\eta) \I (\xi - \eta) \hat{u}(\xi - \eta) (\I \eta)^k \hat{u}(\eta) \overline{(\I \xi)^{k-1} \hat{u}(\xi)}\,\diff\eta\, \diff \xi\\
	&\quad-\iint_{\mathbb{R}^2}m(\xi-\eta,\eta) \hat{u}(\xi - \eta) (\I \eta)^{k+1} \hat{u}(\eta) \overline{(\I \xi)^{k-1} \hat{u}(\xi)}\,\diff\eta\, \diff \xi\\
	&=\colon A_{0}^1+A_{0}^2.
\end{aligned}
\end{equation}
We first calculate the term \(A_{0}^2\). Since \(u\) is real, one has \(\overline{\hat u(\xi)} = \hat u(-\xi)\). Note also that \(m\) is invariant under the map \((\xi, \eta) \mapsto -(\xi, \eta)\). This, together with the additional change of variables \(\xi \leftrightarrow \eta\), shows that
\begin{align*}
	A_{0}^2
	=-\iint_{\mathbb{R}^2}m(\eta-\xi,\xi)\hat{u}(\xi-\eta) (\I \eta)^{k-1} \hat u(\eta) \overline{(\I \xi)^{k+1} \hat u(\xi)} \,\diff\eta \, \diff \xi.
\end{align*}
In light of \eqref{6.1} this yields 
\begin{equation}\label{11}
\begin{aligned}
	A_0+A_{0}^2
	&=\iint_{\mathbb{R}^2}\eta^{-1}\big[m(\xi-\eta,\eta)\eta+m(\eta-\xi,\xi)\xi\big]\\
	&\quad\times \hat{u}(\xi-\eta)(\I \eta)^k \hat{u}(\eta) \overline{(\I \xi)^k \hat{u}(\xi)}\,\diff\eta\, \diff \xi =0.
\end{aligned}
\end{equation}
Hence, it is sufficient to estimate the terms \(A_0^1\) and \(A_j\), \(j=1,\cdots,k-1\). Because of the different singularities 
at different ranges of \(\alpha\) the rest of the proof is divided into two cases.

\subsection*{The case of \(\alpha \in (0,1)\)} Note that the factor \(\I(\xi - \eta) (\I \eta)^k\) appearing in \(A_0^1\) eliminates the low-frequency singularities of \(m(\xi-\eta,\eta)\) because \(k \geq 1\). We are thus left only with estimating the high frequencies. To that aim, 
we first use the triangle inequality to estimate (from \eqref{6.8a}),
\begin{align*}
|m(\xi-\eta,\eta)|
&\lesssim\frac{|\xi-\eta|^{1-\alpha}}{|\eta|}+\frac{|\xi-\eta|^{1-\alpha}+|\xi|^{1-\alpha}}{|\xi-\eta|}\\
&=\frac{|\xi-\eta|^{1-\alpha}}{|\eta|}+\frac{1}{|\xi-\eta|^\alpha}+\frac{|\xi|^{1-\alpha}}{|\xi-\eta|}.
\end{align*}
It follows that the double integral \(A_0^1\) may be \(L^2 \times L^2 \times H^1\)-estimated in modulus by (here, we have used subindices to indicate from which factors in the integral the different factors in the H\"older estimates come from)
\begin{equation}\label{12}
\begin{aligned}
&\bigr\||D|^{1-\alpha} \partial_x u\bigr\|_{L^2_{\xi - \eta}} \bigr\| |D|^{-1} \partial_x^ku \bigr\|_{H^1_\eta}
\bigr\|\partial_x^{k-1}u\bigr\|_{L^2_\xi}\\
&\quad+\bigr\||D|^{-\alpha}\partial_xu\bigr\|_{L^2_{\xi - \eta}}\bigr\|\partial_x^ku\bigr\|_{L^2_\eta} 	\bigr\|\partial_x^{k-1}u\bigr\|_{H^1_\xi}\\
&\quad+\bigr\||D|^{-1} \partial_xu\bigr\|_{H^1_{\xi - \eta}}\bigr\|\partial_x^ku\bigr\|_{L^2_\eta} \bigr\||D|^{1-\alpha}\partial_x^{k-1}u\bigr\|_{L^2_\xi}\\
&\lesssim \|u\|_{H^{2-\alpha}}\|u\|_{H^{k}}\|u\|_{H^k}+\|u\|_{H^{1-\alpha}}\|u\|_{H^{k}}^2  +\|u\|_{H^1}\|u\|_{H^{k}}\|u\|_{H^{k-\alpha}}\\
&\lesssim\|u\|_{H^{2-\alpha}}\|u\|_{H^{k}}^2.
\end{aligned}
\end{equation}
In light of \eqref{10}-\eqref{11} one therefore has
\begin{align}\label{13}
	|A_0|
	\lesssim\|u\|_{H^{2-\alpha}}\|u\|_{H^{k}}^2.
\end{align}
The terms \(A_j\), \(j=1,\ldots,k-1\), are dealt with in a similar fashion. Since \(1 \leq j \leq k -1\) with \(k \geq 1\), the derivatives in \(P(\partial_x^ju,\partial_x^{k-j}u)\) will cancel the unit order low-frequencies singularities in \(m\). The high frequencies may in this case be handled directly with \eqref{6.8a}, and one finds that \(|A_j|\) can be bounded by
\begin{equation}\label{15}
\begin{aligned}
&\bigr\||D|^{1-\alpha}\partial_x^j u\bigr\|_{L^2_{\xi - \eta}} \bigr\||D|^{-1}  \partial_x^{k-j} u\bigr\|_{H^1_\eta} \bigr\|\partial_x^{k} u\bigr\|_{L^2_\xi}\\
&\quad +\bigr\||D|^{-1} \partial_x^j u\bigr\|_{H^1_{\xi - \eta}}\bigr\||D|^{1-\alpha}\partial_x^{k-j} u\bigr\|_{L^2_\eta} \bigr\|\partial_x^{k}u\bigr\|_{L^2_\xi}\\
&\lesssim \|u\|_{H^{j+1-\alpha}}\|u\|_{H^{k-j}}\|u\|_{H^{k}}
 +\|u\|_{H^j}\|u\|_{H^{k+1-j-\alpha}}\|u\|_{H^{k}}\\
&\lesssim\|u\|_{H^{k}}^3.
\end{aligned}
\end{equation}
We conclude from \eqref{13} and \eqref{15} that
\begin{align*}
	|\big(\partial_x^ku, \partial_x^kP(u,u)\big)_2|\lesssim (\|u\|_{H^{2-\alpha}}+\|u\|_{H^{k}})\|u\|_{H^{k}}^2,
\end{align*}
which in view of  \(1 \leq k \leq N\) yields \eqref{eq:cubic estimate} for \(\|u\|_{H^N} < \varepsilon\). This proves the result for \(\alpha \in (0,1)\).

\subsection*{The case of \(\beta = -\alpha \in (0,1)\)}  The triangle inequality applied to \eqref{6.8b} now yields
\begin{equation}\label{15.5}
\begin{aligned}
|m(\xi-\eta,\eta)|
	&\lesssim\frac{|\xi-\eta|^{1-\beta}|\xi|^\beta}{|\eta|^{1-\beta}}
	+\frac{(|\xi|^{1-\beta}+|\xi-\eta|^{1-\beta})|\xi|^\beta}{|\xi-\eta|^{1-\beta}}\\
	&\lesssim \frac{|\xi|}{|\eta|^{1-\beta}}
	+\frac{|\xi|}{|\xi-\eta|^{1-\beta}}+|\xi|^\beta.
\end{aligned}
\end{equation}
Thus
\begin{equation}\label{16}
\begin{aligned}
	|A_{0}^{1}|
	&\lesssim \|u\|_{H^2}\|u\|_{H^{k-1+\beta}}\|u\|_{H^{k}}+\|u\|_{H^{1+\beta}}\|u\|_{H^{k}}^2\\
	&\lesssim\|u\|_{H^2}\|u\|_{H^{k}}^2,
\end{aligned}
\end{equation}
which combined with \eqref{10}--\eqref{11} shows that
\(
	|A_0|
	\lesssim\|u\|_{H^2}\|u\|_{H^{k}}^2.
\)
For the terms \(A_j\), \(j=1,\ldots,k-1\), we now estimate \(m\) symmetrically in \(\eta\) and \(\xi - \eta\).
\begin{equation}\label{18}
\begin{aligned}
|m(\xi-\eta,\eta)|
&\lesssim \left(\frac{|\xi-\eta|^{1-\beta}}{|\eta|^{1-\beta}} +\frac{|\eta|^{1-\beta}}{|\xi-\eta|^{1-\beta}} \right) \left(|\xi-\eta|^\beta+|\eta|^\beta \right)\\
&=\underbrace{\frac{|\xi-\eta|}{|\eta|^{1-\beta}}+\frac{|\xi-\eta|^{1-\beta}}{|\eta|^{1-2\beta}}}_{m_1(\xi-\eta,\eta)} +\underbrace{\frac{|\eta|}{|\xi-\eta|^{1-\beta}} + \frac{|\eta|^{1-\beta}}{|\xi-\eta|^{1-2\beta}}}_{m_1(\eta,\xi-\eta)}.
\end{aligned}
\end{equation}
By the symmetry in \(\xi- \eta\) and \(\eta\), and the same symmetry in the terms \(A_j\), terms expressed as \(m_1(\eta,\xi-\eta)\) can be estimated exactly as the terms \(m_1(\xi-\eta,\eta)\). Note that
\begin{equation}\label{eq:estimate eta < xi - eta}
\frac{|\xi-\eta|^{1-\beta}}{|\eta|^{1-2\beta}} = \frac{|\eta|^\beta}{|\xi - \eta|^\beta}\frac{|\xi-\eta|}{|\eta|^{1-\beta}} \leq \frac{|\xi-\eta|}{|\eta|^{1-\beta}},  
\end{equation}
when \(|\eta| \leq |\xi - \eta|\), while
\begin{equation}\label{eq:estimate eta > xi - eta}
\frac{|\xi-\eta|^{1-\beta}}{|\eta|^{1-2\beta}} =  \frac{|\xi-\eta|^{1-\beta}}{|\eta|^{1-\beta}} |\eta|^{\beta} \leq|\eta|^\beta,
\end{equation}
when \(|\xi - \eta| \leq |\eta|\). Hence, \(m_1(\xi-\eta,\eta) \lesssim \frac{|\xi-\eta|}{|\eta|^{1-\beta}} + |\eta|^\beta\). The monomial arising from the derivatives on \(u\) in the integrand of \(A_j\) is \((\xi-\eta)^j \eta^{k-j} \xi^k\). Excluding the \(\xi^k\) factor for a moment, we thus estimate
\begin{align*}
m_1(\xi-\eta,\eta)  |\xi-\eta|^j |\eta|^{k-j}  &\lesssim |\xi - \eta|^{j+1} |\eta|^{k-j+\beta -1}  + |\xi - \eta|^{j} |\eta|^{k-j+\beta} \\
&\lesssim |\xi - \eta|^{k} |\eta|^{k -1} + |\xi - \eta|^{k-1} |\eta|^{k}.
\end{align*}
It is a direct consequence of this estimate that \(|A_j|\) has a \(\|u\|_{H^k}^3\) bound, as
\begin{equation}\label{18.3}
\begin{aligned}
&\big|\iint_{\mathbb{R}^2}m_1(\xi-\eta,\eta)\mathcal{F}(\partial_x^ju)(\xi-\eta)\mathcal{F}(\partial_x^{k-j}u)(\eta)\overline{\mathcal{F}(\partial_x^{k}u)(\xi)}\,\diff\eta \, \diff \xi\big|\\
& \lesssim  \left(\bigr\||D| \partial_x^j u\bigr\|_{L^2_{\xi-\eta}}\bigr\||D|^{\beta-1}  \partial_x^{k-j}u\bigr\|_{H^1_\eta} +\bigr\|\partial_x^j u\bigr\|_{H^1_{\xi - \eta}}\bigr\||D|^\beta\partial_x^{k-j}u\bigr\|_{L^2_\eta} \right) \bigr\|\partial_x^{k}u\bigr\|_{L^2_\xi}\\
&\lesssim\|u\|_{H^{k}}^3,
\end{aligned}
\end{equation} 
and similarly for \(m_1(\eta, \xi-\eta)\). By combining these estimates with \eqref{16} we obtain
\begin{align*}
	|\big(\partial_x^ku, \partial_x^kP(u,u)\big)_2|\lesssim (\|u\|_{H^2}+\|u\|_{H^{k}})\|u\|_{H^{k}}^2, \quad  \ k\geq1,  
\end{align*}
which finishes the proof.
\end{proof}

\section{The Energy  Estimates} \label{sec:energy estimates}
We now develop commutator estimates to bound the modified energy. If one deals with the operator \(P\) and its symbol \(m\) directly, pointwise monomial estimates like the ones in Section~\ref{sec:modified energy} are not possible (in fact, they are untrue if one wants a bound in \(H^k(\R)\)). One therefore has to perform global transformations. On the physical side, a global transformation is integration by parts; as mentioned, this however corresponds to a local equality on the Fourier side. A different global transformation is change of variables, and it is what we will use to tackle the problem. More precisely, we shall at each step break out the terms of our commutator that we can handle with some pointwise estimates on the Fourier side; whatever remains will be transformed via global transformations, whereafter the resulting integrals will be attacked pointwisely again. A caveat is that the symbol \(m\) displays different behaviours at different frequencies, and must be dealt with accordingly. Simultaneously, changes of variables will in general both change the domain of integration and affect the symmetries of the Fourier variables, so we will utilise the exact form of \(m\) as much as possible to divide \(\R^3\) into symmetric domains. In the end we will end up with a final commutator which, via differences, has two orders of gain in the required Fourier variables. 

We shall work our way via a series of smaller results, ultimately proving the following result.

\begin{proposition}\label{proposition:upper bound  estimates} Let \(\alpha \in (-1,1)\), \(\alpha\neq0\).  Then
	\begin{align}\label{19.5}
	\frac{\diff}{\diff t}E^{(k)}(t)\lesssim \|u\|_{H^2} \|u\|_{H^3} \|u\|_{H^k}^2+\|u\|_{H^k}^4, \quad\, k\geq1.
	\end{align}
\end{proposition}

\subsection{Reduction of \(\frac{\diff}{\diff t} E^{(k)}(t)\)} We start from the modified  (partial) energy \eqref{eq:modified k-energy}, eliminating its highest-order term.

\begin{lemma}\label{lemma:F_0 G_0}
Let  \(F_j = \big(Q(\partial_x^ju,\partial_x^{k+1-j}u),\partial_x^k(-u\partial_xu)\big)_2\), and similarly let 
\(G_j = \big(Q(\partial_x^j u,\partial_x^{k-j}(u\partial_x u)),\partial_x^{k+1}u\big)_2\).
Then 
\[
\frac{\diff}{\diff t}E^{(k)}(t) \eqsim \sum_{j=1}^k c_{k,j} (F_j +G_j),
\] 
where \(c_{k,j}\) are binomial coefficients.
\end{lemma}

\begin{proof}
From the definition of the modified energy \eqref{eq:modified k-energy}, and the equation \eqref{eq:fKdV}, one calculates
\begin{equation}\label{20}
\begin{aligned}
&\frac{1}{2}\frac{\diff}{\diff t}E^{(k)}(t)\\
&=\big(\partial_x^k\partial_tu,\partial_x^ku)_2+(\partial_x^k\partial_tu, \partial_x^kP(u,u)\big)_2
+\big(\partial_x^ku, \partial_x^k\partial_tP(u,u)\big)_2\\
&=\big(\partial_x^k(|D|^\alpha\partial_xu-u\partial_xu),\partial_x^ku\big)_2+\big(\partial_x^k(|D|^\alpha\partial_xu
-u\partial_xu), \partial_x^kP(u,u)\big)_2\\
&\quad+2\big(\partial_x^ku, \partial_x^kP(|D|^\alpha\partial_x u-u\partial_ x u,u)\big)_2\\
&= -\big(\partial_x^ku,\partial_x^k(u\partial_xu)\big)_2 -\big(\partial_x^ku, \partial_x^k(|D|^\alpha\partial_xP(u,u))\big)_2\\
&\quad +\big(\partial_x^k(-u\partial_xu), \partial_x^kP(u,u)\big)_2 +2\big(\partial_x^ku, \partial_x^kP(|D|^\alpha\partial_x u-u\partial_ x u,u)\big)_2,
\end{aligned}
\end{equation} 
where in the second step we have used the symmetry of \(P\), and in the last integration by parts. Now, the bilinear form \(P\) is constructed exactly to satisfy 
\begin{align}\label{23}
-u\partial_xu-|D|^\alpha\partial_xP(u,u)+P(|D|^\alpha\partial_x u,u)+P(u,|D|^\alpha\partial_x u)=0,
\end{align}
see \eqref{5}. Thus, insertion of  \eqref{23} into \eqref{20} yields
\begin{equation}\label{24}
\begin{aligned}
\frac{1}{2}\frac{\diff}{\diff t}E^{(k)}(t)
&=\big(\partial_x^k(-u\partial_xu), \partial_x^kP(u,u)\big)_2 + 2\big(\partial_x^ku, \partial_x^kP(-u\partial_x u,u)\big)_2.
\end{aligned}
\end{equation}
Note here that all the cubic terms in \eqref{20} have been transferred to quartic ones. In the following we will estimate the separate terms on the right-hand side of \eqref{24}. For notational convenience, let \(P=\partial_x Q\) so that
\begin{equation}\label{24.5}
\begin{aligned}
n(\xi-\eta,\eta) &=\frac{-\I}{2\big[|\xi-\eta|^\alpha(\xi-\eta)+|\eta|^\alpha\eta-|\xi|^\alpha\xi\big]},
\end{aligned}
\end{equation}
is the symbol of \(Q\), symmetric in its two arguments (just as \(P\) and \(m\) are). We now decompose \eqref{24} into its highest-order and remainder terms, as
\begin{equation}\label{24.8}
\begin{aligned}
&\big(\partial_x^kP(u,u), \partial_x^k(-u\partial_xu)\big)_2\\
&=\big(\partial_x^{k+1}Q(u,u), \partial_x^k(-u\partial_xu)\big)_2
=2\big(\partial_x^kQ(u,\partial_x u), \partial_x^k(-u\partial_xu)\big)_2\\
&=2\underbrace{\big(Q(u,\partial_x^{k+1}u), \partial_x^k(-u\partial_xu)\big)_2}_{F_0} +
2\sum_{j=1}^k c_{k,j} \underbrace{\big(Q(\partial_x^ju,\partial_x^{k+1-j}u),\partial_x^k(-u\partial_xu)\big)_2}_{F_j},
\end{aligned}
\end{equation}
and, using integration by parts,
\begin{equation}\label{24.9}
\begin{aligned}
&2\big(\partial_x^kP(u,-u\partial_x u),\partial_x^ku\big)_2
=2\big(\partial_x^kQ(u,u\partial_x u),\partial_x^{k+1}u\big)_2\\
&=2\underbrace{\big(Q(u,\partial_x^k(u\partial_x u)),\partial_x^{k+1}u\big)_2}_{G_0}
+2\sum_{j=1}^k c_{k,j} \underbrace{\big(Q(\partial_x^j u,\partial_x^{k-j}(u\partial_x u)),\partial_x^{k+1}u\big)_2}_{G_j}.\\
\end{aligned}
\end{equation}
The solution \(u\)  being real, we have \(\overline{\hat u} = \hat u(-\cdot)\), so the change of variables \(\xi \leftrightarrow \eta\) in combination with basic manipulations shows that 
\begin{align*}
F_0
=-\iint_{\R^2}\overline{n(\eta-\xi,\xi)}\hat{u}(\xi-\eta)\mathcal{F}(\partial_x^k(u\partial_xu))(\eta)
\overline{\mathcal{F}(\partial_x^{k+1}u)(\xi)}\, \diff\eta \diff\xi.
\end{align*}
Because \(n(\xi-\eta,\eta)=\overline{n(\eta-\xi,\xi)}\), we therefore have
\begin{align*}
F_0+G_0
&=\iint_{\R^2}\big[n(\xi-\eta,\eta)-\overline{n(\eta-\xi,\xi)}\big]\hat{u}(\xi-\eta)\\
&\quad \times \mathcal{F}(\partial_x^k(u\partial_xu))(\eta)
\overline{\mathcal{F}(\partial_x^{k+1}u)(\xi)}\, \diff\eta \diff\xi\\
&=0,
\end{align*}
and the result follows.
\end{proof}  

Although the highest-order terms \(F_0\) and \(G_0\) in \(\frac{\diff}{\diff t} E^{(k)}(t)\) completely cancel each other out, there are still the high-order terms \(F_j, G_j\) for \(j \in \{1, k\}\), which in three cases are too high to be estimated separately. The subsequent subsections are devoted to the terms  \(F_j\)  and \(G_j\), \(j=1,2,\ldots, k-1\), where we will have to divide the proof into two cases because of the different behaviour of \(m\) and \(n\) for positive and negative values of \(\alpha\). But we first make away with the case \(k = 1\), and separate further low-order terms out of \(F_1\), \(G_1\) in the case \(k \geq 2\).

\begin{lemma}\label{lemma:k=1}
When \(k=1\),  
\begin{align*}
\frac{\diff}{\diff t}E^{(k)}(t) \lesssim\|u\|_{H^1}^2\|u\|_{H^2}^2.
\end{align*}
\end{lemma}
\begin{proof}
When \(k=1\) there are only  \(F_1\) in \eqref{24.8} and  \(G_1\) in \eqref{24.9} to deal with. From the definition of \(n\) and the estimate \eqref{6.8a} one has  for \(\alpha \in (0,1)\) that
\begin{align*}
|n(\xi-\eta,\eta)|
&\lesssim\frac{1}{|\xi|}\left(\frac{|\xi-\eta|^{1-\alpha}}{|\eta|}+\frac{|\xi-\eta|^{1-\alpha}+|\xi|^{1-\alpha}}{|\xi-\eta|}\right)\\
&=\frac{|\xi-\eta|^{1-\alpha}}{|\xi||\eta|}+\frac{1}{|\xi||\xi-\eta|^\alpha}+\frac{1}{|\xi|^\alpha|\xi-\eta|},
\end{align*}
and for \(\beta=-\alpha \in (0,1)\) that
\begin{align*}
|n(\xi-\eta,\eta)|
&\lesssim\frac{1}{|\xi|^{1-\beta}}\left(\frac{|\xi-\eta|^{1-\beta}}{|\eta|^{1-\beta}}+\frac{|\xi-\eta|^{1-\beta}+|\xi|^{1-\beta}}{|\xi-\eta|^{1-\beta}}\right)\\
&=\frac{|\xi-\eta|^{1-\beta}}{|\xi|^{1-\beta}|\eta|^{1-\beta}}+\frac{1}{|\xi|^{1-\beta}}
+\frac{1}{|\xi-\eta|^{1-\beta}}\\
&\lesssim \frac{1}{|\xi|^{1-\beta}}+\frac{1}{|\eta|^{1-\beta}}
+\frac{1}{|\xi-\eta|^{1-\beta}}.
\end{align*}
Since
\(
F_1 = \big(Q(\partial_x u,\partial_x u),\partial_x (-u\partial_xu)\big)_2
\)
we have for \(\alpha \in (0,1)\) that
\begin{align*}
|F_1|
&\lesssim\bigr\||D|^{1-\alpha} \partial_x u \bigr\|_{L_{\xi-\eta}^2}\bigr\||D|^{-1} \partial_x u\bigr\|_{H_\eta^1} \bigr\| |D|^{-1} \partial_x (u\partial_xu) \bigr\|_{L_\xi^2}\\
&\quad +\bigr\||D|^{-\alpha} \partial_x u \bigr\|_{L_{\xi-\eta}^2}\bigr\| \partial_x u\bigr\|_{H_\eta^1} \bigr\| |D|^{-1} \partial_x (u\partial_xu) \bigr\|_{L_\xi^2}\\
&\quad +\bigr\||D|^{-1} \partial_x u \bigr\|_{L_{\xi-\eta}^2}\bigr\| \partial_x u\bigr\|_{H_\eta^1} \bigr\| |D|^{-\alpha} \partial_x (u\partial_xu) \bigr\|_{L_\xi^2}\\
&\lesssim \|u\|_{H^1}^2 \|u\|_{H^2}^2, 
\end{align*}
and likewise for \(\beta = -\alpha \in (0,1)\): 
\begin{align*}
|F_1|
&\lesssim\bigr\|\partial_x u \bigr\|_{H_{\xi-\eta}^1}\bigr\|\partial_x u\bigr\|_{L_\eta^2} \bigr\| |D|^{\beta -1} \partial_x (u\partial_xu) \bigr\|_{L_\xi^2}\\
&\quad +\bigr\| \partial_x u \bigr\|_{H_{\xi-\eta}^1}\bigr\||D|^{\beta-1}\partial_x u\bigr\|_{L_\eta^2} \bigr\|  \partial_x (u\partial_xu) \bigr\|_{L_\xi^2}\\
&\quad +\bigr\||D|^{\beta -1}\partial_x u \bigr\|_{L_{\xi-\eta}^2}\bigr\| \partial_x u\bigr\|_{H_\eta^1} \bigr\|  \partial_x (u\partial_xu) \bigr\|_{L_\xi^2}\\
&\lesssim \|u\|_{H^1}^2 \|u\|_{H^2}^2.
\end{align*}
For 
\(
G_1 = \big(Q(\partial_x u,(u\partial_x u)),\partial_x^{2}u\big)_2 = {\textstyle \frac{1}{2}} \big(Q(\partial_x u, \partial_x u^2 ),\partial_x^{2}u\big)_2,
\)
the corresponding calculations give
\begin{align*}
|G_1|
&\lesssim\bigr\||D|^{1-\alpha} \partial_x u\bigr\|_{L_{\xi-\eta}^2}\bigr\||D|^{-1}  \partial_x u^2\bigr\|_{H_\eta^1}\bigr\| |D|^{-1}\partial_x^{2}u\bigr\|_{L_\xi^2}\\
&\quad +\bigr\||D|^{-\alpha} \partial_x u\bigr\|_{L_{\xi-\eta}^2}\bigr\| \partial_x u^2 \bigr\|_{H_\eta^1}\bigr\| |D|^{-1}\partial_x^{2}u\bigr\|_{L_\xi^2}\\
&\quad +\bigr\||D|^{-1} \partial_x u\bigr\|_{L_{\xi-\eta}^2}\bigr\| \partial_x u^2 \bigr\|_{H_\eta^1}\bigr\| |D|^{-\alpha}\partial_x^{2}u\bigr\|_{L_\xi^2}\\
&\lesssim  \|u\|_{H^1}^2 \|u\|_{H^2}^2, 
\end{align*}
and
\begin{align*}
|G_1|
&\lesssim\bigr\| \partial_x u\bigr\|_{L_{\xi-\eta}^2}\bigr\|  \partial_x u^2 \bigr\|_{H_\eta^1}\bigr\| |D|^{\beta-1}\partial_x^{2}u\bigr\|_{L_\xi^2}\\
&\quad +\bigr\|\partial_x u\bigr\|_{L_{\xi-\eta}^2}\bigr\||D|^{\beta-1}  \partial_x u^2 \bigr\|_{H_\eta^1}\bigr\| \partial_x^{2}u\bigr\|_{L_\xi^2}\\
&\quad +\bigr\||D|^{\beta-1} \partial_x u\bigr\|_{L_{\xi-\eta}^2}\bigr\| \partial_x u^2 \bigr\|_{H_\eta^1}\bigr\| \partial_x^{2}u\bigr\|_{L_\xi^2}\\
&\lesssim   \|u\|_{H^1}^2 \|u\|_{H^2}^2.
\end{align*}
All taken together, 
\begin{align*}
|F_1| + |G_1|\lesssim\|u\|_{H^1}^2\|u\|_{H^2}^2.
\end{align*}
Since \(\frac{1}{2}\frac{\diff}{\diff t}E^{(1)}(t) = 2(F_0 + G_0) + 2(F_1 + G_1)\), with \(F_0 + G_0 = 0\), the proposition follows.
\end{proof}

When \(k\geq2\), we employ commutator estimates to handle the  terms \(F_1\), \(G_1\), \(F_k\) and \(G_k\). We first extract the lower-order parts of these terms.
Using integration by parts  we decompose \(F_1\) as 
\begin{equation}\label{eq:def F_{1,0}}
\begin{aligned}
F_1&=-\big(Q(\partial_x u,\partial_x^{k}u),\partial_x^k(u\partial_xu)\big)_2
=\big(P(\partial_x u,\partial_x^{k}u),\partial_x^{k-1}(u\partial_xu)\big)_2\\
&=\underbrace{\big(uP(\partial_x u,\partial_x^{k}u),\partial_x^{k}u\big)_2}_{F_{1,0}}
+\sum_{l=1}^{k-1}c_{k-1,l} \underbrace{\big(P(\partial_x u,\partial_x^{k}u),\partial_x^lu\partial_x^{k-l}u\big)_2}_{F_{1,l}}.
\end{aligned}
\end{equation}
To eliminate the low-frequency singularities, we decompose \(G_1\) instead in the following manner: 
\begin{equation}\label{eq:def G_{1,0}}
\begin{aligned}
G_1&=\big(Q(\partial_x u,\partial_x^{k-1}(u\partial_x u)),\partial_x^{k+1}u\big)_2
=\big(P(\partial_x u,\partial_x^{k-1}(u\partial_x u)),-\partial_x^{k}u\big)_2\\
&=\underbrace{\big(P(\partial_x u,\partial_x (u\partial_x^{k-1}u)),-\partial_x^{k}u\big)_2}_{G_{1,0}} + 
\sum_{l=1}^{k-2}c_{k-2,l} \underbrace{\big(P(\partial_x u,\partial_x (\partial_x^lu\partial_x^{k-l-1}u)),-\partial_x^{k}u\big)_2}
_{G_{1,l}}.
\end{aligned}
\end{equation}
By the symmetry of \(P\) and \(Q\), we also have
\[
F_1=F_k,
\]
whence it is natural to define \(F_{k,l} := F_{1,l}\). Note, however, that
\begin{equation}\label{eq:G_k}
G_1 \neq G_k=\big(Q(u\partial_x u,\partial_x^ku),\partial_x^{k+1}u\big)_2.
\end{equation}
It follows from the following lemma that the difficult terms left to treat are \(F_{1,0}\), \(G_{1,0}\) and \(G_k\).

\begin{lemma}\label{lemma:ddt sim I} Let  \(k \geq 2\). Then, in the language of \eqref{eq:def F_{1,0}}--\eqref{eq:G_k},
\[
\frac{\diff}{\diff t} E^{(k)}(t) \eqsim (F_{1,0} + G_{1,0}) + (F_{k,0} + G_k) + R,
\]
with \(R \lesssim \| u \|_{H^k}^4\).
\end{lemma}

\begin{proof}
We estimate the terms \(F_j\), \(G_j\) for \(j=2,\ldots,k-1\), \(F_{1,l} = F_{k,l}\) for \(l=1,\ldots,k-1\), and \(G_{1,l}\) for \(l=1,\ldots,k-2\). The estimates follow the lines of those in the proof of Lemma~\ref{lemma:k=1}, so we give the details only for \(F_j\) and \(G_j\). 

Integrating by parts, one has \(F_j=(P(\partial_x^ju,\partial_x^{k+1-j}u),\partial_x^{k-1}(u\partial_xu))_2\)
and \(G_j=-(P(\partial_x^ju,\partial_x^{k-j}(u\partial_x u)),\partial_x^ku )_2\). In view of \eqref{6.8a}, we can estimate 
\begin{equation}\label{25}
\begin{aligned}
|F_j|
&\lesssim\bigr\||D|^{1-\alpha} \partial_x^j u\bigr\|_{L_{\xi-\eta}^2}\bigr\||D|^{-1} \partial_x^{k+1-j}u\bigr\|_{H_\eta^1}
\bigr\| \partial_x^{k} u^2 \bigr\|_{L_\xi^2}\\
&\quad+\bigr\||D|^{-1} \partial_x^j u\bigr\|_{H_{\xi-\eta}^1}\bigr\||D|^{1-\alpha} \partial_x^{k+1-j}u\bigr\|_{L_\eta^2}\bigr\|  \partial_x^{k} u^2 \bigr\|_{L_\xi^2}\\
&\lesssim \|u\|_{H^{j+1-\alpha}}\|u\|_{H^{k+1-j}}\|u\|_{H^{k}}^2
+\|u\|_{H^j}\|u\|_{H^{k+2-j-\alpha}}\|u\|_{H^{k}}^2,
\end{aligned}
\end{equation}
when \(2 \leq j \leq k-1\) and \(\alpha \in (0,1)\). Similarly,
\begin{equation}\label{28}
\begin{aligned}
|G_j|
&\lesssim\bigr\||D|^{1-\alpha} \partial_x^ju \bigr\|_{L_{\xi-\eta}^2}\bigr\||D|^{-1} \partial_x^{k-j+1} u^2 \bigr\|_{H_\eta^1}
\bigr\| \partial_x^{k}u \bigr\|_{L_\xi^2}\\
&\quad+\bigr\||D|^{-1} \partial_x^j u \bigr\|_{H_{\xi-\eta}^1}\bigr\||D|^{1-\alpha} \partial_x^{k-j+1} u^2 \bigr\|_{L_\eta^2}
\bigr\| \partial_x^{k}u \bigr\|_{L_\xi^2}\\
&\lesssim \|u\|_{H^{j+1-\alpha}}\|u\|_{H^{k+1-j}}^2\|u\|_{H^{k}}
+\|u\|_{H^j}\|u\|_{H^{k+2-j-\alpha}}^2\|u\|_{H^{k}},
\end{aligned}
\end{equation}
for the same range of parameters. When \(\beta=-\alpha \in (0,1)\), it instead follows from  \eqref{18}--\eqref{eq:estimate eta > xi - eta}, that 
\begin{equation}\label{54}
\begin{aligned}
|F_j|
&\lesssim\bigr\||D| \partial_x^j u\bigr\|_{L_{\xi-\eta}^2}\bigr\||D|^{\beta-1} \partial_x^{k+1-j} u\bigr\|_{H_\eta^1}
\bigr\| \partial_x^{k} u^2 \bigr\|_{L_\xi^2}\\
&\quad+\bigr\||D|^\beta \partial_x^j u \bigr\|_{L_{\xi-\eta}^2} \bigr\| \partial_x^{k+1-j} u \bigr\|_{H_\eta^1}
\bigr\| \partial_x^{k} u^2 \bigr\|_{L_\xi^2}\\
&\quad+\bigr\| \partial_x^j u\bigr\|_{H_{\xi-\eta}^1}\bigr\||D|^\beta \partial_x^{k+1-j} u \bigr\|_{L_\eta^2}
\bigr\| \partial_x^{k} u^2 \bigr\|_{L_\xi^2}\\
&\quad+\bigr\||D|^{\beta-1} \partial_x^j u\bigr\|_{H_{\xi-\eta}^1} \bigr\| |D| \partial_x^{k+1-j} u \bigr\|_{L_\eta^2}
\bigr\| \partial_x^{k} u^2 \bigr\|_{L_\xi^2}\\
&\lesssim \|u\|_{H^{j+1}}\|u\|_{H^{k+1-j+\beta}}\|u\|_{H^{k}}^2+
\|u\|_{H^{j+\beta}}\|u\|_{H^{k+2-j}}\|u\|_{H^{k}}^2,
\end{aligned}
\end{equation}
and
\begin{equation}\label{56}
\begin{aligned}
|G_j|
&\lesssim\bigr\| |D| \partial_x^j u \bigr\|_{L_{\xi-\eta}^2} \bigr\||D|^{\beta-1} \partial_x^{k-j+1} u^2 \bigr\|_{H_\eta^1} \bigr\| \partial_x^{k}u\bigr\|_{L_\xi^2}\\
&\quad+\bigr\||D|^\beta \partial_x^j u\bigr\|_{L_{\xi-\eta}^2} \bigr\| \partial_x^{k-j+1} u^2 \bigr\|_{H_\eta^1} \bigr\| \partial_x^{k} u\bigr\|_{L_\xi^2}\\
&\quad+\bigr\| \partial_x^j u \bigr\|_{H_{\xi-\eta}^1} \bigr\||D|^\beta \partial_x^{k-j+1} u^2\bigr\|_{L_\eta^2}
\bigr\| \partial_x^{k} u \bigr\|_{L_\xi^2}\\
&\quad+\bigr\||D|^{\beta-1} \partial_x^j u \bigr\|_{H_{\xi-\eta}^1} \bigr\||D| \partial_x^{k-j+1} u^2 \bigr\|_{L_\eta^2} \bigr\| \partial_x^{k} u\bigr\|_{L_\xi^2}\\
&\lesssim \|u\|_{H^{j+1}}\|u\|_{H^{k+1-j+\beta}}^2\|u\|_{H^{k}}+
\|u\|_{H^{j+\beta}}\|u\|_{H^{k+2-j}}^2\|u\|_{H^{k}}.
\end{aligned}
\end{equation}
All of these terms may be estimated by \(\|u\|_{H^{k}}^4\). In view of \eqref{6.8a}--\eqref{6.8b} and \eqref{18}--\eqref{eq:estimate eta > xi - eta} one can estimate, in the same fashion as \eqref{25}--\eqref{56}, the terms \(F_{1,l}\) and \(G_{1,l}\) to show that
\begin{align}\label{59}
|\sum_{l=1}^{k-1}F_{1,l} |+|\sum_{l=1}^{k-2}G_{1,l} |
\lesssim \|u\|_{H^{k}}^4.
\end{align}
Since Lemma~\ref{lemma:F_0 G_0} holds that \(\frac{\diff}{\diff t}E^{(k)}(t) \eqsim \sum_{j=1}^k c_{k,j} (F_j +G_j)\), this completes the proof.
\end{proof}

\subsection{Higher-order estimates: \(F_{1,0} + G_{1,0}\)}\label{subsec:F10 + G10}
In this subsection we will deal with the term \(F_{1,0} + G_{1,0}\).
Set
\begin{align*}
\mathcal{A}_1:=\{(\xi,\eta,\sigma)\in\R^3:\min\{|\xi|,|\eta|,|\sigma|\}< 1\}.
\end{align*}
We first estimate the integrals \(F_{1,0}\) and \(G_{1,0}\)  (separately) when restricted to the domain of integration \(\mathcal{A}_1\). With slight abuse of notation, we let \(\mathcal{A}_1 G_{1,0}\) denote this restriction of the integral, and similarly for other integrals to come. 

\begin{lemma}\label{lemma:A_1 estimates}
The low-frequency integrals \(\mathcal{A}_1 F_{1,0}\) and \(\mathcal{A}_1 G_{1,0}\) satisfy
\[
| \mathcal{A}_1 F_{1,0} | + | \mathcal{A}_1 G_{1,0} | \lesssim \|u\|_{H^2}^2\|u\|_{H^k}^2.
\]
\end{lemma}
\begin{proof}
The projection \(\mathcal{A}_1 G_{1,0}\) equals
\[
\int_{\mathcal{A}_1}m(\xi-\eta,\eta) \eta (\xi-\eta) (\I \sigma)^{k-1} (-\I \xi)^k  \diff Q(\hat u),
\]
where
\begin{equation}\label{eq:dQ}
\diff Q(\hat u) := \hat{u}(\xi-\eta)\hat{u}(\eta-\sigma) \hat{u}(\sigma)\overline{\hat{u}(\xi)} \, \diff\xi \, \diff\eta \, \diff\sigma.
\end{equation}
is the quartic (signed) measure appearing in the integral. Note that the variables here are \(\xi\), \(\sigma\), \(\xi -\eta\) and \(\eta - \sigma\) (but not, for example, \(\eta\)). In the set \(\mathcal{A}_1\), one however has
\[
|\xi| + |\eta| + |\sigma| \lesssim 1 +  |\xi - \eta| + |\eta - \sigma|,
\]
and
\begin{equation}\label{eq:xi eta estimate}
|\xi \eta| \lesssim 1 + |\xi - \eta| +  |\xi - \eta||\eta - \sigma| +  |\eta - \sigma|^2,
\end{equation}
as a consequence of the triangle inequality applied repeatedly to the different cases. We can thus estimate the total symbol pointwise in the following way: for \(\alpha \in (0,1)\),
\begin{align*}
\left| m(\xi-\eta,\eta) \eta (\xi-\eta) \right| &\eqsim \left( |\xi - \eta|^{2-\alpha} +  |\eta|^{2-\alpha} \right)\\
&\lesssim \left( 1 + |\xi - \eta|^{2-\alpha} +  |\eta - \sigma|^{2-\alpha} \right),
\end{align*}
and for \(\beta = -\alpha  \in (0,1)\), 
\begin{align*}
\left| m(\xi-\eta,\eta) \eta (\xi-\eta) \right| &\eqsim |\xi \eta|^\beta \left(|\xi - \eta|^{2-2\beta}  +  |\eta|^{2-2\beta}  \right) |\xi - \eta|^\beta\\
&\lesssim \left( 1 + |\xi - \eta|^{\beta} + |\xi - \eta|^{\beta} |\eta - \sigma|^\beta +  |\eta - \sigma|^{2\beta} \right) \\
&\quad \times \left( 1 + |\xi - \eta|^{2-2\beta} + |\eta - \sigma|^{2-2\beta} \right)  |\xi - \eta|^{\beta}\\
&\lesssim 1 + |\xi - \eta|^2 + |\eta - \sigma|^2 + |\xi - \eta|^2 |\eta - \sigma|^2.
\end{align*}
Multiplying with \(|\sigma|^{k-1} |\xi|^k\), we may use one \(H^1\)-estimate on the \(\sigma\)-terms, and one on either the \((\eta-\sigma)\)-terms or the \((\xi-\eta)\)-terms, to obtain that
\[
| \mathcal{A}_1 G_{1,0} | \lesssim \|u\|_{H^2}^2\|u\|_{H^k}^2.
\] 
Similarly, \(\mathcal{A}_1 F_{1,0}\) equals
\begin{equation}\label{eq:A1 F10}
\int_{\mathcal{A}_1}m(\eta-\sigma,\sigma) \I (\eta-\sigma) (\I \sigma)^k 
(-\I \xi)^k  \diff Q(\hat u),
\end{equation}
so the estimate for \(\alpha \in (0,1)\) can be obtained exactly as for \(G_{1,0}\) if we borrow a \(\sigma\) from \(\sigma^k\), yielding
\begin{align*}
\left| m(\eta-\sigma,\sigma) (\eta- \sigma) \sigma \right| &\lesssim \left( 1 + |\xi - \eta|^{2-\alpha} +  |\eta - \sigma|^{2-\alpha} \right).
\end{align*}
When \(\beta = - \alpha \in (0,1)\), the symbol is also the same as for \(G_{1,0}\) (with \((\xi - \eta, \eta)\) replaced by \((\eta - \sigma, \sigma)\)) so we just need to replace \eqref{eq:xi eta estimate} with its symmetric equivalence
\begin{equation*}
|\sigma \eta| \lesssim 1 + |\eta - \sigma| +  |\eta - \sigma||\xi - \eta| +  |\xi - \eta|^2,
\end{equation*}
to obtain that
\begin{align*}
\left| m(\eta-\sigma,\sigma) (\eta- \sigma) \sigma \right| &\lesssim 1 + |\xi - \eta|^2 + |\eta - \sigma|^2 + |\xi - \eta|^2 |\eta - \sigma|^2,
\end{align*}
in the case \(\beta = -\alpha  \in (0,1)\). Multiplying with \(|\sigma|^{k-1} |\xi|^k\) and performing the estimates with one sup-norm on \(\||D|^{k-1} u\|_{H^1_\sigma}\) and one on either the \((\eta-\sigma)\)-factor or the \((\xi-\eta)\)-factor, we obtain the desired bound on \(\mathcal{A}_1 F_{1,0}\).
\end{proof}

As what concerns \(F_{1,0} + G_{1,0}\) we are thus left with estimating the high-frequency part of the integrals. Let \(\mathcal{A}_1^c\) denote the complement of \(\mathcal{A}_1\), that is 
\[
\mathcal{A}_1^c = \{(\xi,\eta,\sigma)\in\R^3 \colon |\xi|,|\eta|,|\sigma| \geq  1\},
\]
and \(\mathcal{A}_1^c F_{1,0}\) and  \(\mathcal{A}_1^c G_{1,0}\) the corresponding restrictions of the integrals \(F_{1,0}\) and \(G_{1,0}\) to the set \(\mathcal{A}_1^c\). We divide \(\mathcal{A}_1^c \subset \R^3\) further using
\begin{equation}\label{eq:A2}
\begin{aligned}
\mathcal{A}_2:=\{(\xi,\eta,\sigma)\in {\mathcal A}_1^c  \colon &{\textstyle \frac{1}{10}} |z_2| < |z_1-z_2|+|z_2-z_3|,\\ &\quad \text{for some choice of } z_j  = \xi,\eta,\sigma\},
\end{aligned}
\end{equation}
and its complement \(\mathcal{A}_2^c\) in \(\mathcal{A}_1^c\). The point of \(\mathcal{A}_2\) is that, by the triangle inequality, any powers of \(\xi\), \(\eta\) and \(\sigma\) can be readily transferred to powers of \(|\xi - \eta| + |\eta - \sigma|\), as we have
\[
|z_1| + |z_2| + |z_3| \lesssim |z_1 - z_2| + |z_2| + |z_2 - z_3| \lesssim  |z_1 - z_2| + |z_2 - z_3|.
\]
As before, we use the sets \(\mathcal{A}_2\) and \(\mathcal{A}_2^c\) also to denote restrictions of integrals to the same sets. Using the possibility to move derivatives to \(\xi - \eta\) and \(\eta - \sigma\) as just described, it is straightforward to verify the following proposition.

\begin{proposition}\label{prop:A_2}
\[
| {\mathcal A}_2 F_{1,0} | + | {\mathcal A}_2 G_{1,0} | \lesssim \|u\|_{H^2}^2\|u\|_{H^k}^2.
\]
\end{proposition}
Now, the integral left to treat is
\begin{equation}\label{62.2}
\begin{aligned}
&\mathcal{A}_2^c (G_{1,0}+F_{1,0})\\ 
&= -\I \int_{\mathcal{A}_2^c}  \left[ m(\xi-\eta,\eta) (\xi - \eta) \eta -  m(\eta-\sigma,\sigma) (\eta - \sigma) \sigma  \right]  \sigma^{k-1}  \xi^k \diff Q(\hat u).
\end{aligned}
\end{equation}
We note here that in \(\mathcal{A}_2^c\) one has \(|\xi/\eta - 1| + |1 - \sigma/\eta| \leq \frac{1}{10}\), so that there
\[
\sign(\xi)=\sign(\sigma)=\sign(\eta),
\]
which will be of great help. In particular, it allows us to assume that \(\xi, \eta, \sigma\) are all positive, as \(m\), \(u\) and the energy are real: because \(m\) is jointly even in its arguments, one gets after the shift of variables \((\xi, \eta, \sigma) \rightarrow -(\xi, \eta, \sigma)\) that the integrals in \eqref{62.2} equal twice their value when taken only over
\[
\mathcal{A}_{2,+}^c := \left\{ (\xi,\eta, \sigma) \in \mathcal{A}_2^c \colon \xi, \eta, \sigma \geq 1\right\},
\]
in view of that \(\mathcal{A}_2^c\) lies in the exterior of the unit cube. In \(\mathcal{A}_{2,+}^c\) we may furthermore write 
\begin{equation}\label{eq:mu nu}
\xi=(1+\mu)\eta \quad\text{ and }\quad \sigma=(1+\nu)\eta,
\end{equation}
where \(|\mu|,|\nu|\leq {\textstyle \frac{1}{10}}\) are uniformly small quantities. Consequently,
\begin{equation}\label{eq:similar}
\xi\eqsim\sigma\eqsim\eta\eqsim\xi-\eta+\sigma \gtrsim 1.
\end{equation}
We shall use the small variables \(\mu = \frac{\xi - \eta}{\eta}\) and \(\nu = \frac{\sigma - \eta}{\eta}\) to move derivatives from \(\eta \eqsim \xi \eqsim\sigma\) to \(\xi - \eta\) and \(\eta - \sigma\). 

Pointwise estimates in \eqref{62.2} are not sufficient for the desired bounds. We will therefore apply changes of variables to take advantage of the commutator structure of \eqref{62.2}. Although we make us of the difference in this integral, we emphasise that this difference is in fact (implicitly) inherent already in each of the terms \(F_{1,0}\) and \(G_{1,0}\). As the following proposition makes precise, these two terms are namely equal in \(\mathcal{A}_2^c\), modulo a good term, so that \(\mathcal{A}_2^c F_{1,0} \sim \frac{1}{2} \mathcal{A}_2^c (F_{1,0} + G_{1,0})\).

\begin{lemma}\label{lemma:F10 sim G10} 
In \eqref{62.2} one has
\(
\mathcal{A}_2^c F_{1,0} = \mathcal{A}_2^c  G_{1,0} + \bigO(\|u\|_{H^2} \|u\|_{H^3} \|u\|_{H^k}^2).
\)
\end{lemma}

\begin{proof}
First note that \(m(\sigma - \eta, \eta) \eta (\sigma - \eta) = m(\eta - \sigma, \sigma) \sigma (\eta - \sigma)\). We now apply the changes of variables \(\eta \leftrightarrow \xi \leftrightarrow \sigma \leftrightarrow \eta\), after which we use the above equality and the fact that \(\mathcal{A}_2^c G_{1,0}\) and \(u\) are real. That yields
\begin{align*}
\mathcal{A}_2^c G_{1,0} &= -\I \int_{\mathcal{A}_2^c}m(\xi-\eta,\eta) \eta (\xi-\eta) \sigma^{k-1} \xi^k \hat{u}(\xi-\eta)\hat{u}(\eta-\sigma) \hat{u}(\sigma)\overline{\hat{u}(\xi)} \, \diff\xi \, \diff\eta \, \diff\sigma\\
&= -\I \int_{\mathcal{A}_2^c}m(\sigma-\eta,\eta) \eta (\sigma - \eta) \xi^{k-1} \sigma^k \hat{u}(\sigma-\eta)\hat{u}(\eta-\xi) \hat{u}(\xi)\overline{\hat{u}(\sigma)} \, \diff\xi \, \diff\eta \, \diff\sigma\\
&= \I \int_{\mathcal{A}_2^c}m(\eta-\sigma,\sigma) \sigma (\eta - \sigma) \xi^{k-1} \sigma^k \hat{u}(\eta-\sigma)\hat{u}(\xi-\eta) \hat{u}(\sigma)\overline{\hat{u}(\xi)} \, \diff\xi \, \diff\eta \, \diff\sigma.
\end{align*}
By comparing with the symbol for \(F_{1,0}\) in \eqref{eq:A1 F10}, we see that
\[
{\mathcal{A}}_2^c (F_{1,0} - G_{1,0}) = \I \int_{\mathcal{A}_2^c}m(\eta-\sigma,\sigma) (\eta - \sigma) (\xi - \sigma) \xi^{k-1} \sigma^k \diff Q(\hat u).
\]
When \(\alpha\) is positive the total symbol may be estimated by
\begin{align*}
&\left( \frac{|\eta - \sigma|^{1-\alpha}}{|\sigma|} + \frac{|\sigma|^{1-\alpha}}{|\eta - \sigma|} \right) |\eta - \sigma| \left( |\xi - \eta| + |\eta - \sigma| \right) |\xi|^{k-1} |\sigma|^k\\
&\lesssim |\eta - \sigma|^{2-\alpha} |\xi -\eta| |\xi|^{k-1} |\sigma|^{k-1} + |\eta - \sigma|^{3-\alpha} |\xi|^{k-1} |\sigma|^{k-1}\\ 
&\quad+  |\xi -\eta| |\xi|^{k} |\sigma|^{k-\alpha} + |\eta - \sigma| |\xi|^{k} |\sigma|^{k-\alpha},
\end{align*}
where we have used the equivalence \(\sigma \eqsim \xi\) in \({\mathcal{A}}_2^c\). Direct \((L^2)^2 \times (H^1)^2\)-estimates then yields a \(H^2 \times H^3 \times (H^k)^2\)-bound on \(u\) in this case. When \(\alpha = - \beta\) is negative, we again use the equivalence \(\sigma \eqsim \xi \eqsim \eta\) to obtain
\begin{equation}\label{eq:Lemma4.7**}
\begin{aligned}
& |\eta|^{\beta} \left( \frac{|\eta - \sigma|^{1-\beta}}{|\sigma|^{1-\beta}} + \frac{|\sigma|^{1-\beta}}{|\eta - \sigma|^{1-\beta}} \right) |\eta - \sigma| \left( |\xi - \eta| + |\eta - \sigma| \right) |\xi|^{k-1} |\sigma|^k\\
&\lesssim |\eta - \sigma|^{2-\beta} |\xi -\eta| |\xi|^{k-1+\beta} |\sigma|^{k-1+\beta} +|\eta - \sigma|^{3-\beta}  |\xi|^{k-1+\beta} |\sigma|^{k-1+\beta} \\ 
&\quad+ |\eta - \sigma|^\beta  |\xi -\eta| |\xi|^{k} |\sigma|^{k} + |\eta - \sigma|^{1+\beta} |\xi|^{k} |\sigma|^{k}.
\end{aligned}
\end{equation}
By further estimating the second term in \eqref{eq:Lemma4.7**} by
\begin{align*}
&|\eta - \sigma|^{3-\beta} |\xi|^{k-1+\beta} |\sigma|^{k-1+\beta}\\
&\lesssim |\eta - \sigma|^{2}(|\eta|^{1-\beta}+|\sigma|^{1-\beta}) |\xi|^{k-1+\beta} |\sigma|^{k-1+\beta}\\
&\lesssim |\eta - \sigma|^{2}|\xi|^{k} |\sigma|^{k-1+\beta},
\end{align*}
we may obtain a \(H^2 \times H^3 \times (H^k)^2\)-bound by applying the \(H^1\)-estimates to the factors associated with \((\xi - \eta)\) and \((\eta - \sigma)\).
\end{proof}

Now, we separate two comparable parts of \(\mathcal{A}_2^c F_{1,0}\) and \(\mathcal{A}_2^c G_{1,0}\). By expressing \(\mathcal{A}_2^c F_{1,0}\) as
\begin{align*}
&\I \int_{\mathcal{A}_2^c} m(\eta-\sigma,\sigma) (\eta - \sigma) \sigma^k \xi^k \, \diff Q(\hat u)\\ 
&= \I \int_{\mathcal{A}_2^c}\frac{m(\eta-\sigma,\sigma)}{\eta} (\eta - \xi) (\eta - \sigma) \sigma^k \xi^k \, \diff Q(\hat u)\\ 
&\quad+ \I \int_{\mathcal{A}_2^c} \frac{m(\eta-\sigma,\sigma)}{\eta}  (\eta - \sigma) \sigma^k \xi^{k+1} \, \diff Q(\hat u)\\
&=: \mathcal{A}_2^c F_{1,0}^{(\eta - \xi)} + \mathcal{A}_2^c \tilde F_{1,0},
\end{align*}
and \(\mathcal{A}_2^c G_{1,0}\) as
\begin{align*}
&-\I \int_{\mathcal{A}_2^c} m(\xi-\eta,\eta) \eta (\xi-\eta) \sigma^{k-1} \xi^k  \diff Q(\hat u)\\
&= -\I \int_{\mathcal{A}_2^c}  m(\xi-\eta,\eta) (\eta - \sigma) (\xi-\eta) \sigma^{k-1} \xi^k  \diff Q(\hat u)\\
&\quad -\I \int_{\mathcal{A}_2^c}  \frac{m(\xi-\eta,\eta)}{\xi} (\xi-\eta) \sigma^{k} \xi^{k+1}  \diff Q(\hat u)\\
&=: \mathcal{A}_2^c G_{1,0}^{(\eta - \sigma)} + \mathcal{A}_2^c \tilde G_{1,0},
\end{align*}
we may namely neglect the terms \(\mathcal{A}_2^c F_{1,0}^{(\eta - \xi)}\) and \(\mathcal{A}_2^c G_{1,0}^{(\eta - \sigma)}\) in relation to the others as these both contain a second-order difference product \((\eta - \xi) (\eta - \sigma)\). Using the same type of symbol estimates as in the proof of  Lemma~\ref{lemma:F10 sim G10} one then easily obtains \(\|u\|_{H^2}^2 \|u\|_{H^k}^2\)-bounds for these good terms.   We summarise the situation in the following proposition.    

\begin{proposition}\label{prop:tilde FG}
We have \(\mathcal{A}_2^c (G_{1,0} + F_{1,0}) = \mathcal{A}_2^c (\tilde G_{1,0} + \tilde F_{1,0}) + \bigO(\|u\|_{H^2}^2 \|u\|_{H^k}^2)\), where
\begin{equation}\label{eq:tilde G + F}
\begin{aligned}
&\mathcal{A}_2^c (\tilde G_{1,0} + \tilde F_{1,0})\\ 
&= -\I \int_{\mathcal{A}_2^c} \left[ \frac{m(\xi-\eta,\eta)}{\xi} (\xi-\eta)  - \frac{m(\eta-\sigma,\sigma)}{\eta}  (\eta - \sigma)  \right]  \sigma^{k} \xi^{k+1}  \diff Q(\hat u).
\end{aligned}
\end{equation}
\end{proposition}

Recall from the discussion after Proposition~\ref{prop:A_2} that, by symmetry, it is enough to consider positive values of \(\xi\eqsim\sigma\eqsim\eta\eqsim\xi-\eta+\sigma \gtrsim 1\) in \(\mathcal{A}_{2}^c\); that set is called \(\mathcal{A}_{2,+}^c\). Although it is invariant under changes of variables between the variables \(\xi\), \(\eta\) and \(\sigma\), it is not under the transformation
\[
 \eta\mapsto \xi-\eta+\sigma,
 \]
 that we shall now use. This change of variables substitutes \(\xi - \eta\) for \(\eta - \sigma\) and leaves \(\diff Q(\hat u)\) invariant. It furthermore maps \(\eta\) into a variable of comparable size in \(\mathcal{A}_{2,+}^c\). It however maps \(\mathcal{A}_{2,+}^c\) to \(\mathcal{B} =  \mathcal{B}_1 \cap {\mathcal B}_2\), where
\begin{equation*}\label{eq:B1}
\begin{aligned}
\mathcal{B}_1:=\{(\xi,\eta,\sigma)\in\R^3 \colon \xi,\xi-\eta+\sigma,\sigma \geq 1\},
\end{aligned}
\end{equation*}
and
 \begin{equation*}\label{eq:B2}
\begin{aligned}
\mathcal{B}_2 :=\big\{(\xi,\eta,\sigma)\in\R^3 \colon &|\sigma-\eta|+|\xi-\sigma|\leq {\textstyle\frac{1}{10}} \xi,\\ 
&|\xi-\eta|+|\sigma-\xi|\leq {\textstyle\frac{1}{10}} \sigma,\\
&|\eta-\sigma|+|\xi-\eta|\leq {\textstyle\frac{1}{10}} (\xi-\eta + \sigma) \big\}.
\end{aligned}
\end{equation*}
Performing the above change of variables on the second term in \eqref{eq:tilde G + F}, we obtain that we are left with estimating
\begin{equation}\label{eq:J-K II}
\begin{aligned}
&\I  \int_{\mathcal{A}_{2,+}^c \cap \mathcal{B}} \left[\frac{m(\xi-\eta,\eta)}{\xi}-\frac{m(\xi-\eta,\sigma)}{\xi-\eta+\sigma}\right]  (\xi - \eta) \sigma^k \xi^{k+1} \, \diff Q(\hat u)\\
&\quad + \I  \int_{\mathcal{A}_{2,+}^c \setminus \mathcal{B}} \frac{m(\xi-\eta,\eta)}{\xi}  (\xi - \eta) \sigma^k \xi^{k+1} \, \diff Q(\hat u)\\
&\quad -\I  \int_{\mathcal{B} \setminus \mathcal{A}_{2,+}^c} \frac{m(\xi-\eta,\sigma)}{\xi-\eta+\sigma}  (\xi - \eta) \sigma^k \xi^{k+1} \diff Q(\hat u).
\end{aligned} 
\end{equation}

The two latter terms may be easily done away with using the following lemma.
\begin{lemma}\label{lemma:sets}
The sets \(\mathcal{A}_{2,+}^c \setminus \mathcal{B}\) and \(\mathcal{B} \setminus \mathcal{A}_{2,+}^c\) are both contained in
\[
\left\{(\xi, \eta, \sigma) \in \R^3 \colon \textstyle{\frac{1}{2}} \leq \xi, \eta, \sigma \leq \textstyle{\frac{3}{2}} + 30(|\xi - \eta| + |\eta - \sigma|) \right\}.
\]
\end{lemma} 
\begin{proof}
Start by assuming that \((\xi, \eta, \sigma) \in \mathcal{A}_{2,+}^c\) but do not belong to \(\mathcal{B}_1\). Then
\[
\xi - \eta + \sigma < 1 \quad \text{ and }\qquad \xi \geq 1.
\]
But \(|\frac{z_i}{z_j} - 1| \leq \frac{1}{10}\) for any choice of \(z_i, z_j \in \{\xi, \eta, \sigma\}\), so \(|\eta - \sigma| \leq \frac{1}{5} \xi\). Hence \(1 \leq \xi \leq \frac{5}{4}\), and from that similar bounds for \(\eta\) and \(\sigma\) follow by their relative closeness to \(\xi\).

If instead \((\xi, \eta, \sigma) \in \mathcal{A}_{2,+}^c\) but do not belong to \(\mathcal{B}_2\), one first shows by relative closeness that \(\frac{1}{2} \eta \leq \xi - \eta + \sigma \leq \frac{3}{2} \eta\), so that \(\frac{1}{10} (\xi - \eta + \sigma)\) can be replaced, up to a small factor, by \(\eta\) in the definition of \(\mathcal{B}_2\). Since \((\xi, \eta, \sigma)\) now does \emph{not} belong to this set, we are back in the situation of \eqref{eq:A2} but with an adjustment of the factor \(\frac{1}{10}\), that is
\[
 {\textstyle \frac{1}{20}} z_2 \leq |z_1-z_2|+|z_2-z_3|
\]
holds for some choice of \(z_j  \in \{\xi,\eta,\sigma\}\), \(z_i \neq z_j\) for \(i \neq j\). The relative closeness of \(\xi, \sigma, \eta\) then yield the stated inequality.

The proof for the case when \((\xi, \eta, \sigma) \in \mathcal{B}\) but do not belong to \(\mathcal{A}_{2,+}^c\) is carried out in almost exactly the same fashion, using the relative closeness of variables, and that it is only \(\eta\) that differ in the inequalities (with respect to \(\xi - \eta + \sigma\)).
\end{proof}

\begin{corollary}\label{cor:thin sets}
The integrals over \(\mathcal{A}_{2,+}^c \setminus \mathcal{B}\) and \(\mathcal{B} \setminus \mathcal{A}_{2,+}^c\) in \eqref{eq:J-K II} can be estimated in modulus by a factor of \(\|u\|_{H^{2}} \|u\|_{H^{3}} \|u\|_{H^k}^2\).
\end{corollary}
\begin{proof}
This is an almost immediate consequence of Lemma~\ref{lemma:sets} and the estimates \eqref{6.8a} and \eqref{6.8b} for \(m\) as any powers of \(\xi\) and \(\sigma\) can be estimated by the same powers of \(1 + |\xi - \eta| + |\eta - \sigma|\).
\end{proof}

\noindent \textbf{The main commutator.} We are now at our final step of Section~\ref{subsec:F10 + G10}, where we estimate the first integral in \eqref{eq:J-K II}. This is a commutator that will improve our estimates by two orders of cancellations, allowing us to move derivatives in an advantageous way. 
\begin{proposition}\label{prop:G1 + F1 full}
\[
|F_{1,0} + G_{1,0} | \lesssim \|u\|_{H^2} \|u\|_{H^3} \|u\|_{H^k}^2.
\]
\end{proposition}

\begin{remark}
It is an immediate corollary of Proposition~\ref{prop:G1 + F1 full} that with \(F_{k,0}\) and \(G_k\) as in~\eqref{eq:def F_{1,0}}--\eqref{eq:G_k} one has 
\[
\left| {\textstyle \frac{1}{4}\frac{\diff}{\diff t}} E^{(k)}(t) - (F_{k,0} + G_k) \right| \lesssim  \|u\|_{H^2} \|u\|_{H^3} \|u\|_{H^k}^2 + \|u\|_{H^k}^4.
\]
\end{remark}

\begin{proof}
Denote	
\begin{align*}
N(\xi,\eta,\sigma)\colon
=\frac{m(\xi-\eta,\eta)}{\xi}
-\frac{m(\xi-\eta,\sigma)}{\xi-\eta+\sigma}.
\end{align*} 
In view of \eqref{eq:similar} one then calculates that
\begin{align*}
&N(\xi,\eta,\sigma)=\frac{1}{2\big[|\xi-\eta|^\alpha(\xi-\eta)+\eta^{1+\alpha}- \xi^{1+\alpha}\big]}\\
&\quad\quad\quad\quad\quad-\frac{1}{2\big[|\xi-\eta|^\alpha(\xi-\eta)+\sigma^{1+\alpha} - (\xi-\eta+\sigma)^{1+\alpha} \big]}\\
&=\frac{\sigma^{1+\alpha} - (\xi-\eta+\sigma)^{1+\alpha} - (\eta^{1+\alpha} - \xi^{1+\alpha})}{2\big[|\xi-\eta|^\alpha(\xi-\eta)+\eta^{1+\alpha} - \xi^{1+\alpha} \big]\big[|\xi-\eta|^\alpha(\xi-\eta)+\sigma^{1+\alpha} - (\xi-\eta+\sigma)^{1+\alpha} \big]}.
\end{align*}
If we let
\[
Q(\xi,\eta,\sigma)=\sigma^{1+\alpha} - (\xi-\eta+\sigma)^{1+\alpha} - (\eta^{1+\alpha} - \xi^{1+\alpha})
\]
be the numerator above, it follows from the definition of \(m\) that
\begin{align*}
|N(\xi,\eta,\sigma)|\lesssim \frac{|m(\xi-\eta,\eta) m(\xi-\eta,\sigma)|}{\xi (\xi-\eta+\sigma)}|Q(\xi,\eta,\sigma)|.
\end{align*}
The fractions in this expression may be estimated using Proposition~\ref{proposition:multiplier 1} and \eqref{eq:similar}. For \( \alpha \in (0,1)\),
\begin{equation}\label{67.9}
\begin{aligned}
&\frac{|m(\xi-\eta,\eta) m(\xi-\eta,\sigma)|}{\xi (\xi-\eta+\sigma)}\\
&\eqsim \frac{1}{\xi (\xi-\eta+\sigma)}\left(\frac{|\xi-\eta|^{1-\alpha}}{\eta}+\frac{\eta^{1-\alpha}}{|\xi-\eta|}\right)\left(\frac{|\xi-\eta|^{1-\alpha}}{\sigma}+\frac{\sigma^{1-\alpha}}{|\xi-\eta|}\right)\\
&\lesssim\underbrace{|\xi-\eta|^{2-2\alpha}\xi^{-4}}_{R_1^\alpha(\xi,\eta)}+\underbrace{|\xi-\eta|^{-\alpha}\xi^{-2-\alpha}}_{R_2^\alpha(\xi,\eta)}+\underbrace{|\xi-\eta|^{-2}\xi^{-2\alpha}}_{R_3^\alpha(\xi,\eta)},
\end{aligned}
\end{equation}
and when \(\beta=-\alpha \in (0,1)\),
\begin{equation}\label{67.99}
\begin{aligned}
&\frac{|m(\xi-\eta,\eta) m(\xi-\eta,\sigma)|}{\xi (\xi-\eta+\sigma)}\\
&\eqsim \frac{1}{\xi^{1-\beta}(\xi-\eta+\sigma)^{1-\beta}}\left(\frac{|\xi-\eta|^{1-\beta}}{\eta^{1-\beta}}+\frac{\eta^{1-\beta}}{|\xi-\eta|^{1-\beta}}\right)\left(\frac{|\xi-\eta|^{1-\beta}}{\sigma^{1-\beta}}+\frac{\sigma^{1-\beta}}{|\xi-\eta|^{1-\beta}}\right)\\
&\lesssim\underbrace{|\xi-\eta|^{2-2\beta}\xi^{-4+4\beta}}_{R_1^\beta(\xi,\eta)}+\underbrace{\xi^{-2+2\beta}}_{R_2^\beta(\xi,\eta)}+\underbrace{|\xi-\eta|^{-2+2\beta}}_{R_3^\beta(\xi,\eta)}.
\end{aligned}
\end{equation}
The symbol \(Q\) may be readily dealt with using the proximity expressed in \eqref{eq:mu nu}, as
\[
Q(\xi,\eta,\sigma)=\eta^{1+\alpha}\Big\{\big[(1+\nu)^{1+\alpha}-(1+\nu+\mu)^{1+\alpha}\big]-\big[1-(1+\mu)^{1+\alpha}\big]\Big\}.
\]
A direct Taylor expansion yields
\begin{align*}
(1+\nu)^{1+\alpha}-(1+\nu+\mu)^{1+\alpha} &=-(1+\alpha)\mu(1+\nu)^{\alpha}+O(\mu^2)\\ 
&=-(1+\alpha)\mu + O(\mu^2 + |\mu\nu|).
\end{align*} 
Since, similarly,
\begin{align*}
1-(1+\mu)^{1+\alpha}=-(1+\alpha)\mu+O(\mu^2),
\end{align*} 
we have
\begin{align*}
\big|(1+\nu)^{1+\alpha}-(1+\nu+\mu)^{1+\alpha} - 1 +(1+\mu)^{1+\alpha}\big|   \lesssim |\mu| (|\mu| + |\nu|),
\end{align*} 
and thus
\begin{equation}\label{67.9999}
\begin{aligned}
|Q(\xi,\eta,\sigma)| &\lesssim |\eta|^{-1+\alpha}|\xi-\eta|(|\xi-\eta|+|\eta-\sigma|).
\end{aligned}
\end{equation}
In effect, we have moved two derivatives from \(\xi \eqsim \sigma\) to \(\xi - \eta\) and \(\eta - \sigma\).
It then readily follows from \eqref{eq:similar}-\eqref{67.9999} and the triangle inequality that
\begin{align*}
R_1^\alpha(\xi,\eta)|Q(\xi,\eta,\sigma)|&\lesssim |\xi-\eta|^{4-2\alpha}\xi^{-5+\alpha}+|\xi-\eta|^{3-2\alpha}|\eta-\sigma| \xi^{-5+\alpha}\\
&\lesssim |\xi-\eta|^{-\alpha} \xi^{-1},
\end{align*}
\begin{align*}
R_2^\alpha(\xi,\eta)|Q(\xi,\eta,\sigma)|&\lesssim |\xi-\eta|^{2-\alpha} \xi^{-3}+|\xi-\eta|^{1-\alpha} |\eta-\sigma|\xi^{-3}\\
&\lesssim |\xi-\eta|^{-\alpha} \xi^{-1},
\end{align*}
\begin{align*}
R_3^\alpha(\xi,\eta)|Q(\xi,\eta,\sigma)|\lesssim \xi^{-1-\alpha}+|\xi-\eta|^{-1}|\eta-\sigma| \xi^{-1-\alpha},
\end{align*}
\begin{align*}
R_1^\beta(\xi,\eta)|Q(\xi,\eta,\sigma)|&\lesssim |\xi-\eta|^{4-2\beta} \xi^{-5+3\beta}+|\xi-\eta|^{3-2\beta}|\eta-\sigma| \xi^{-5+3\beta}\\
&\lesssim |\xi-\eta|^{\beta} \xi^{-1},
\end{align*}
\begin{align*}
R_2^\beta(\xi,\eta)|Q(\xi,\eta,\sigma)|
&\lesssim |\xi-\eta|^2 \xi^{-3+\beta}+|\xi-\eta||\eta-\sigma| \xi^{-3+\beta}\\
&\lesssim |\xi-\eta|^{\beta} \xi^{-1},
\end{align*}
and
\begin{align*}
R_3^\beta(\xi,\eta)|Q(\xi,\eta,\sigma)|&\lesssim |\xi-\eta|^{2\beta} \xi^{-1-\beta}+|\xi-\eta|^{-1+2\beta}|\eta-\sigma| \xi^{-1-\beta}\\
&\lesssim |\xi-\eta|^{\beta} \xi^{-1}+|\xi-\eta|^{-1+\beta}|\eta-\sigma| \xi^{-1}.
\end{align*}
This is all to be multiplied with \( (\xi - \eta) \sigma^k \xi^{k+1} \) in \eqref{eq:J-K II}, whence, finally, it follows that 
\begin{align}\label{eq:J-K estimate}
| \mathcal{A}_2^c (\tilde G_{1,0} + \tilde F_{1,0}) | \lesssim \|u\|_{H^{2}}\|u\|_{H^{3}}\|u\|_{H^k}^2.
\end{align} 
Now, tracing back, we have
\[
F_{1,0} + G_{1,0} = \mathcal{A}_1(F_{1,0} + G_{1,0}) + \mathcal{A}_1^c (F_{1,0} + G_{1,0}),
\]
where the first term is \(O(\|u\|_{H^k}^4)\) in view of Lemma~\ref{lemma:A_1 estimates} (recall that \(k \geq 2\)) ; the second is divided into 
\[
\mathcal{A}_1^c (F_{1,0} + G_{1,0}) =  \mathcal{A}_2 (F_{1,0} + G_{1,0}) + \mathcal{A}_2^c (F_{1,0} + G_{1,0}),
\]
where again Proposition~\ref{prop:A_2} shows that the first is \(O(\|u\|_{H^k}^4)\). Finally, Proposition~\ref{prop:tilde FG} shows that 
\[
\mathcal{A}_2^c (F_{1,0} + G_{1,0}) \eqsim \mathcal{A}_2^c (\tilde F_{1,0} + \tilde G_{1,0}) + O(\|u\|_{H^k}^4).
\]
In view of Corollary~\ref{cor:thin sets} this finalises the proof.
\end{proof}

\subsection{The remaining term: \(G_k\)} In this short subsection we will treat the last term in our energy estimate from Lemma~\ref{lemma:ddt sim I}. Note that the single term \(F_{k,0} = F_{1,0}\) has already been effectively estimated: according to Lemma~\ref{lemma:A_1 estimates} its low frequencies may be estimated by the appropriate term in \((H^2)^2 \times (H^k)^2\), and Proposition~\ref{prop:A_2} shows that the projection of the integral onto the set \(\mathcal{A}_2\) obeys the same bound; now, by Lemma~\ref{lemma:F10 sim G10}, \(F_{k,0} \eqsim F_{k,0} + G_{k,0}\) modulo terms in \(H^2 \times H^3 \times (H^k)^2\) in \(\mathcal{A}_2^c\), so Proposition~\ref{prop:G1 + F1 full} implies that
\[
|F_{k,0}| \lesssim \|u\|_{H^2}\|u\|_{H^3} \|u\|_{H^k}^2.
\]
It remains only to bound \(G_{k}\).

\begin{lemma}\label{lemma:G_k}
The term \(G_k\) from \eqref{eq:G_k} satisfies
\begin{align*}
|G_k|\lesssim\|u\|_{H^2}\|u\|_{H^3}\|u\|_{H^k}^2.
\end{align*}
\end{lemma}
\begin{proof}
We first move some derivatives in \(G_k\) using an integration-by-parts type argument:
\begin{align*}
G_k&= \frac{1}{2} (Q( \partial_x u^2,\partial_x^ku),\partial_x^{k+1}u)_2\\
&= \frac{1}{2} \int n(\xi-\eta,\eta) (\xi - \eta) \eta^k \xi^{k+1} \widehat{u^2}(\xi - \eta) \hat u(\eta) \overline{\hat u(\xi)} \, \diff\eta\, \diff\xi\\
&= \frac{1}{2} \int n(\xi-\eta,\eta) (\xi - \eta)^2 \eta^k \xi^{k} \widehat{u^2}(\xi - \eta) \hat u(\eta) \overline{\hat u(\xi)} \, \diff\eta\, \diff\xi\\
&\quad+ \frac{1}{2} \int n(\xi-\eta,\eta) (\xi - \eta) \eta^{k+1} \xi^{k} \widehat{u^2}(\xi - \eta) \hat u(\eta) \overline{\hat u(\xi)} \, \diff\eta\, \diff\xi\\
&= \frac{1}{2} \int n(\xi-\eta,\eta) (\xi - \eta)^2 \eta^k \xi^{k} \widehat{u^2}(\xi - \eta) \hat u(\eta) \overline{\hat u(\xi)} \, \diff\eta\, \diff\xi\\
&\quad - \frac{1}{2} (Q( \partial_x u^2,\partial_x^ku),\partial_x^{k+1}u)_2,
\end{align*}
where in the last equaility we have taken advantage of the anti-symmetry of \(n\) in \((\xi,\eta)\), and the fact that \(u\) is real whereas \(n\) is imaginary. Thus
\[
G_k = \frac{1}{4} \int n(\xi-\eta,\eta) (\xi - \eta)^2 \eta^k \xi^{k} \widehat{u^2}(\xi - \eta) \hat u(\eta) \overline{\hat u(\xi)} \, \diff\eta\, \diff\xi.
\]
Now, combining \eqref{6.8a} with \eqref{24.5} one has  for \(0<\alpha<1\),
\begin{align*}
|n(\xi-\eta,\eta)|
&\lesssim\frac{1}{|\xi|} \left(\frac{|\xi-\eta|^{1-\alpha}}{|\eta|}+\frac{|\xi-\eta|^{1-\alpha}+|\xi|^{1-\alpha}}{|\xi-\eta|} \right)\\
&\lesssim\frac{1}{|\xi|^\alpha |\eta|} + \frac{1}{|\xi||\eta|^{\alpha}}+\frac{1}{|\xi||\xi-\eta|^\alpha}+\frac{1}{|\xi|^\alpha|\xi-\eta|},
\end{align*}
and, using \eqref{6.8b} for \(0<\beta=-\alpha<1\), that
\begin{align*}
|n(\xi-\eta,\eta)|
&\lesssim\frac{1}{|\xi|^{1-\beta}}\left(\frac{|\xi-\eta|^{1-\beta}}{|\eta|^{1-\beta}}+\frac{|\xi-\eta|^{1-\beta}+|\xi|^{1-\beta}}{|\xi-\eta|^{1-\beta}}\right)\\
&\lesssim\frac{1}{|\eta|^{1-\beta}}+\frac{1}{|\xi|^{1-\beta}}
+\frac{1}{|\xi-\eta|^{1-\beta}}.
\end{align*}
This yields the desired estimate in the same way as in the rest of the paper.
\end{proof}

\section{Proof of the main result}\label{sec:final}
We finally give the proof of our main theorem.
\begin{proof}[Proof of Theorem~\ref{thm:main}] In view of \eqref{19.5}, summing over \(k\) from 1 to \(N\), one has
	\begin{align*}
	\sum_{k=1}^NE^{(k)}(t)\lesssim \sum_{k=1}^NE^{(k)}(0)+\int_0^t\|u(s,\cdot)\|_{H^N}^4\,\diff s,
	\end{align*}
	which  in turn yields 
	\begin{equation}\label{83}
	\sum_{k=1}^NE^{(k)}(t)+\|u\|_{L^2}^2\lesssim \sum_{k=1}^NE^{(k)}(0)+\|u_0\|_{L^2}^2
	+\int_0^t\|u(s,\cdot)\|_{H^N}^4\,\diff  s.
	\end{equation}
Here,  we have used the \(L^2\)-conservation of solutions to \eqref{eq:fKdV}.	
	According to Lemma~\ref{lemma:lower bound  estimates}, we on the other hand have
	\begin{align}\label{84}
	\sum_{k=1}^NE^{(k)}(t)+\|u\|_{L^2}^2
	\eqsim \frac{1}{2}\|u\|_{H^N}^2
	\end{align}
for all \(t \geq 0\) and all \(\|u\|_{H^N} < \varepsilon\) sufficiently small. We conclude from \eqref{83}-\eqref{84} that 
	\begin{align*}
	\|u\|_{H^N}^2\lesssim \|u_0\|_{H^N}^2+\int_0^t\|u(s,\cdot)\|_{H^N}^4\,\diff s,
	\end{align*}
	which finishes the proof by a continuity argument. 
\end{proof}

\section*{Acknowledgement} Part of this research was carried out at the Yau Mathematical Sciences Center, Tsinghua University. The authors are indebted to P. Yu for his kind hospitality and interesting mathematical discussions. The authors would also like to thank J.-C. Saut for valuable comments on an earlier version of this manuscript. The authors are grateful to the two referees, and for their comments and suggestions that helped improve the exposition of the paper.


\begin{thebibliography}{10}
	
	\bibitem{MR3460636}
	{\sc T.~Alazard and J.-M. Delort}, {\em Sobolev estimates for two dimensional
		gravity water waves}, Ast\'erisque,  (2015), pp.~viii+241.
	
	\bibitem{MR2599457}
	{\sc J.~Biello and J.~K. Hunter}, {\em Nonlinear {H}amiltonian waves with
		constant frequency and surface waves on vorticity discontinuities}, Comm.
	Pure Appl. Math., 63 (2010), pp.~303--336.
	
	\bibitem{MR1555257}
	{\sc G.~D. Birkhoff}, {\em On the periodic motions of dynamical systems}, Acta
	Math., 50 (1927), pp.~359--379.
	
	\bibitem{MR3248030}
	{\sc A.~Bressan and K.~T. Nguyen}, {\em Global existence of weak solutions for
		the {B}urgers-{H}ilbert equation}, SIAM J. Math. Anal., 46 (2014),
	pp.~2884--2904.
	
	\bibitem{MR2727172}
	{\sc A.~Castro, D.~C\'ordoba, and F.~Gancedo}, {\em Singularity formations for
		a surface wave model}, Nonlinearity, 23 (2010), pp.~2835--2847.
	
	\bibitem{MR3445499}
	{\sc W.~Craig and C.~Sulem}, {\em Normal form transformations for
		capillary-gravity water waves}, in Hamiltonian partial differential equations
	and applications, vol.~75 of Fields Inst. Commun., Fields Inst. Res. Math.
	Sci., Toronto, ON, 2015, pp.~73--110.
	
	\bibitem{MR3502161}
	\leavevmode\vrule height 2pt depth -1.6pt width 23pt, {\em Mapping properties
		of normal forms transformations for water waves}, Boll. Unione Mat. Ital., 9
	(2016), pp.~289--318.
	
	\bibitem{EW16}
	{\sc M.~Ehrnstr{\"o}m and E.~Wahl\'en}, {\em On {W}hitham's conjecture of a
		highest cusped wave for a nonlocal shallow water wave equation}, Ann. Inst.
	H. Poincar{\'e} Anal. Non. Lin{\'e}aire, online first at
	https://doi.org/10.1016/j.anihpc.2019.02.006.
	
	\bibitem{MR3103170}
	{\sc G.~Fonseca, F.~Linares, and G.~Ponce}, {\em The {IVP} for the dispersion
		generalized {B}enjamin-{O}no equation in weighted {S}obolev spaces}, Ann.
	Inst. H. Poincar\'e Anal. Non Lin\'eaire, 30 (2013), pp.~763--790.
	
	\bibitem{MR2993751}
	{\sc P.~Germain, N.~Masmoudi, and J.~Shatah}, {\em Global solutions for the
		gravity water waves equation in dimension 3}, Ann. of Math. (2), 175 (2012),
	pp.~691--754.
	
	\bibitem{MR3094592}
	{\sc R.~Grimshaw and D.~Pelinovsky}, {\em Global existence of small-norm
		solutions in the reduced {O}strovsky equation}, Discrete Contin. Dyn. Syst.,
	34 (2014), pp.~557--566.
	
	\bibitem{MR3625189}
	{\sc B.~Harrop-Griffiths, M.~Ifrim, and D.~Tataru}, {\em Finite depth gravity
		water waves in holomorphic coordinates}, Ann. PDE, 3 (2017), pp.~Art. 4, 102.
	
	\bibitem{MR2754070}
	{\sc S.~Herr, A.~D. Ionescu, C.~E. Kenig, and H.~Koch}, {\em A
		para-differential renormalization technique for nonlinear dispersive
		equations}, Comm. Partial Differential Equations, 35 (2010), pp.~1827--1875.
	
	\bibitem{MR2982741}
	{\sc J.~K. Hunter and M.~Ifrim}, {\em Enhanced life span of smooth solutions of
		a {B}urgers-{H}ilbert equation}, SIAM J. Math. Anal., 44 (2012),
	pp.~2039--2052.
	
	\bibitem{MR3535894}
	{\sc J.~K. Hunter, M.~Ifrim, and D.~Tataru}, {\em Two dimensional water waves
		in holomorphic coordinates}, Comm. Math. Phys., 346 (2016), pp.~483--552.
	
	\bibitem{MR3348783}
	{\sc J.~K. Hunter, M.~Ifrim, D.~Tataru, and T.~K. Wong}, {\em Long time
		solutions for a {B}urgers-{H}ilbert equation via a modified energy method},
	Proc. Amer. Math. Soc., 143 (2015), pp.~3407--3412.
	
	\bibitem{MR3291137}
	{\sc V.~M. Hur and L.~Tao}, {\em Wave breaking for the {W}hitham equation with
		fractional dispersion}, Nonlinearity, 27 (2014), pp.~2937--2949.
	
	\bibitem{MR3667289}
	{\sc M.~Ifrim and D.~Tataru}, {\em The lifespan of small data solutions in two
		dimensional capillary water waves}, Arch. Ration. Mech. Anal., 225 (2017),
	pp.~1279--1346.
	
	\bibitem{MR2291918}
	{\sc A.~D. Ionescu and C.~E. Kenig}, {\em Global well-posedness of the
		{B}enjamin-{O}no equation in low-regularity spaces}, J. Amer. Math. Soc., 20
	(2007), pp.~753--798.
	
	\bibitem{MR3552008}
	{\sc A.~D. Ionescu and F.~Pusateri}, {\em Global analysis of a model for
		capillary water waves in two dimensions}, Comm. Pure Appl. Math., 69 (2016),
	pp.~2015--2071.
	
	\bibitem{MR2267286}
	{\sc T.~Kappeler and P.~Topalov}, {\em Global wellposedness of {K}d{V} in
		{$H^{-1}(\Bbb T,\Bbb R)$}}, Duke Math. J., 135 (2006), pp.~327--360.
	
	\bibitem{2018arXiv180204851K}
	{\sc R.~{Killip} and M.~{Visan}}, {\em {KdV is wellposed in $H^{-1}$}}.
	\newblock Accepted for publication in Ann. of Math. (2019). arXiv:1802.04851.
	
	\bibitem{MR2455893}
	{\sc A.~Kiselev, F.~Nazarov, and R.~Shterenberg}, {\em Blow up and regularity
		for fractal {B}urgers equation}, Dyn. Partial Differ. Equ., 5 (2008),
	pp.~211--240.
	
	\bibitem{MR3317254}
	{\sc C.~Klein and J.-C. Saut}, {\em A numerical approach to blow-up issues for
		dispersive perturbations of {B}urgers' equation}, Phys. D, 295/296 (2015),
	pp.~46--65.
	
	\bibitem{MR1976047}
	{\sc H.~Koch and N.~Tzvetkov}, {\em On the local well-posedness of the
		{B}enjamin-{O}no equation in {$H^s({\Bbb R})$}}, Int. Math. Res. Not.,
	(2003), pp.~1449--1464.
	
	\bibitem{MR3188389}
	{\sc F.~Linares, D.~Pilod, and J.-C. Saut}, {\em Dispersive perturbations of
		{B}urgers and hyperbolic equations {I}: {L}ocal theory}, SIAM J. Math. Anal.,
	46 (2014), pp.~1505--1537.
	
	\bibitem{MR2970711}
	{\sc L.~Molinet and D.~Pilod}, {\em The {C}auchy problem for the
		{B}enjamin-{O}no equation in {$L^2$} revisited}, Anal. PDE, 5 (2012),
	pp.~365--395.
	
	\bibitem{MR3906854}
	{\sc L.~Molinet, D.~Pilod, and S.~Vento}, {\em On well-posedness for some
		dispersive perturbations of {B}urgers' equation}, Ann. Inst. H. Poincar\'{e}
	Anal. Non Lin\'{e}aire, 35 (2018), pp.~1719--1756.
	
	\bibitem{Poincare}
	{\sc H.~Poincar{\'e}}, {\em {New Methods of Celestial Mechanics}}, no.~13 in
	History of Modern Physics and Astronomy, American Institute of Physics, New
	York, 1993.
	
	\bibitem{MR533234}
	{\sc J.-C. Saut}, {\em Sur quelques g\'en\'eralisations de l'\'equation de
		{K}orteweg-de {V}ries}, J. Math. Pures Appl. (9), 58 (1979), pp.~21--61.
	
	\bibitem{MR803256}
	{\sc J.~Shatah}, {\em Normal forms and quadratic nonlinear {K}lein-{G}ordon
		equations}, Comm. Pure Appl. Math., 38 (1985), pp.~685--696.
	
	\bibitem{wang2016global}
	{\sc X.~Wang}, {\em Global regularity for the 3{D} finite depth capillary water
		waves}, arXiv:1611.05472,  (2016).
	
\end{thebibliography}
\end{document}